\documentclass[12pt,xcolor=svgnames]{amsart}

\usepackage{amsmath, amsthm, amssymb, amsfonts, verbatim}

\usepackage{xspace,mathtools,bm}

\usepackage[svgnames]{xcolor}

\usepackage[pagebackref=true, colorlinks=true, citecolor=blue]{hyperref}

\hypersetup{
    colorlinks=true,
    citecolor=red,
    linkcolor=blue,
    filecolor=magenta,      
    urlcolor=cyan,
    pdfpagemode=FullScreen,
    }
\newtheorem{theorem}{Theorem}[section]
\newtheorem{corollary}[theorem]{Corollary}
\newtheorem{lemma}[theorem]{Lemma}
\newtheorem{definition}[theorem]{Definition}
\newtheorem{proposition}[theorem]{Proposition}
\newtheorem{remark}[theorem]{Remark}

\newcommand{\R}{\mathbb{R}}
\newcommand{\eps}{\varepsilon}
\newcommand{\abs}[1]{\mid\!#1\!\mid}
\newcommand{\norm}[2]{\left\lVert#1\right\rVert_{#2}}
\newcommand{\eqdef}{:=}

\newcommand{\Obsolete}[1]{
    }

\date{\today}

\begin{document}

\title
    [A Nonlocal Multi-species Aggregation-Diffusion Equation]
    {Global Existence for a Nonlocal Multi-species Aggregation-Diffusion Equation}

\author{Elaine Cozzi}
\address{Department of Mathematics, Oregon State University}
\curraddr{}
\email{cozzie@math.oregonstate.edu}

\author{Zachary Radke}
\address{Department of Mathematics, Oregon State University}
\curraddr{}
\email{radkeza@oregonstate.edu}

\maketitle

\begin{abstract}
We consider the question of global existence of smooth solutions to a multi-species aggregation-diffusion equation for a class of singular interaction kernels.  We establish a smallness condition on the initial data which yields global existence of smooth solutions.  We also give conditions on the species interaction which ensure that pointwise inequalities comparing species densities are preserved by the evolution.       
\end{abstract}

\section{Introduction}
Aggregation-diffusion equations have been an active area of study in the field of mathematical biology for the last several decades.  A widely studied example is the scalar aggregation-diffusion equation of the form
\begin{align}\label{aggregation equation}
\tag{$AG_{\nu}$}
\begin{cases}
\frac{\partial \rho}{\partial t} = \nu \Delta \rho - \nabla \cdot (u\rho), \hspace{.2cm} &(x,t) \in \R^d \times (0,T]\\
u =  \nabla(\mathcal{K}\ast \rho), \hspace{.2cm} &(x,t) \in \R^d \times (0,T]\\
\rho \big|_{t=0} = \rho_0, \hspace{.2cm} &x \in \R^d.
\end{cases}
\end{align}
In the system \eqref{aggregation equation}, the density of a single species $\rho(x,t)$ undergoes both diffusion and non-local self-attraction or repulsion due to the advection term $\nabla \cdot (u\rho)$.  The function $\mathcal{K}$ is a spatial averaging kernel.  From the constitutive law $u =  \nabla(\mathcal{K}\ast \rho)$, we see that $\mathcal{K}$ governs the self-attraction or repulsion of the species.  For this reason, we refer to $\mathcal{K}$ as the {\em interaction kernel}.  

The single-species aggregation equation (\ref{aggregation equation}) has been extensively studied for various kernels in both the diffusive and non-diffusive setting (see for example \cite{CKY, Perthame, Horstmann, Bedrossian, LiRodrigo, BRB, BC, JV} and references therein).  Often, one assumes that the kernel $\mathcal{K}$ is the negative of the fundamental solution of the Laplacian, or the Newtonian potential.  In this setting, when $\nu>0$, (\ref{aggregation equation}) represents the well-known Patlak-Keller-Segel system modeling chemotaxis \cite{KS, Perthame, Horstmann, BRB, BDP, HP}.  

One can generalize \eqref{aggregation equation} to the multi-species setting via the system of equations
\begin{align}\label{PDE}
\tag{$MSAG_{\nu}$}
\begin{cases}
\frac{\partial \rho_i}{\partial t} = {\nu}_i \Delta \rho_i - \nabla\cdot(u_i\rho_i), \hspace{.2cm}  &(x,t) \in \R^d \times (0,T]\\
u_i =  \nabla \left(\sum_{j=1}^N h_{ij}\mathcal{K}_{ij}\ast \rho_j\right), \hspace{.2cm} &(x,t) \in \R^d \times (0,T]\\
\rho_i \big|_{t=0} = \rho_{i,0} , \hspace{.2cm} &x\in \R^d.\\
\end{cases}
\end{align}
Here $\rho_1,\hdots, \rho_N$ are the densities of $N$ species, ${\nu}_i>0$ is the diffusion coefficient of species $\rho_i$, and $\mathcal{K}_{ij}$ is an interaction kernel; to elaborate, for each pair $i,j$, one can interpret $\mathcal{K}_{ij}\ast \rho_j$ as a sensing function, measuring the degree to which species $i$ senses species $j$.  Here, species $i$ moves in the direction of greatest increase (if $h_{ij}>0$) or decrease (if $h_{ij}<0$) of the sensing function.  In this way, the sign of $h_{ij}$ determines whether species $i$ is attracted to or repelled by species $j$, and the magnitude of $h_{ij}$ determines the strength of attraction or repulsion.       

The generalized system (\ref{PDE}) can be written in condensed vector form, given by
\begin{align}\label{PDEcondensed}
\begin{cases}
\frac{\partial \rho}{\partial t} = \nu\Delta \rho - \nabla \cdot (u\rho), \hspace{.2cm} &(x,t) \in \R^d \times (0,T] \\
u = \nabla ((H\circ\mathcal{K}) \ast \rho),  \hspace{.2cm} &(x,t) \in \R^d \times (0,T]\\
\rho\big|_{t=0}=\rho_0, \hspace{.2cm} &x \in \R^d, 
\end{cases}
\end{align}
where $\rho = (\rho_1,\cdots,\rho_N)^T$, $\nu=\text{diag}({\nu}_i)$ is a diagonal matrix consisting of the viscosities, $H$ is a matrix with ${ij}$-entry $h_{ij}$, and $\mathcal{K}$ is a matrix with ${ij}$-entry $\mathcal{K}_{ij}$. Here $H\circ \mathcal{K}$ denotes the Hadamard product of $H$ and $\mathcal{K}$,
$$(H\circ \mathcal{K})_{ij} = h_{ij}\mathcal{K}_{ij}.$$
The equality $u = \nabla ((H\circ \mathcal{K}) \ast \rho)$ is understood component-wise by $$u_i = \nabla ((H\circ\mathcal{K}) \ast \rho)_i := \nabla \left(\sum_{j=1}^N h_{ij}\mathcal{K}_{ij}\ast \rho_j\right).$$  Throughout the paper, we refer to $H$ as the {\em interaction matrix}.

In recent years, nonlocal models such as (\ref{PDE}) have been commonly used to model species interaction \cite{GHLP, MainGHLP, JPZ, CSS, WS, PPH, BH}, as organisms often move in response to stimuli outside of their local surroundings, such as distant sounds or smells emanating from other species or food sources.  Both the organism's ability to detect a given stimulus and its response to the stimulus will likely vary depending on both the type of stimulus and its distance from the organism.  This phenomenon is captured in the varying interaction kernels $\mathcal{K}_{ij}$ and the entries of the interaction matrix $H$.   

Many recent results in this area address the question of short-time and global in time existence of both weak and classical solutions, under various assumptions on the interaction kernels, interaction matrix, or both.  We provide details on a few recent results most relevant to our work.  In \cite{MainGHLP, GHLP}, the authors establish global existence of non-negative weak solutions in any dimension and smooth solutions in one dimension under the assumption that the interaction kernels are twice differentiable and have bounded gradient.  These works place no assumption on the interaction matrix $H$ or on the relationship between the interaction kernels.  On the other hand, in \cite{JPZ}, the authors prove global existence of weak solutions assuming the interaction kernels are merely integrable, but place a so-called detailed balance condition on the kernels; in particular, they require that there exist constants $\pi_1,\pi_2, .... \pi_N$ such that for all $1\leq i,j\leq N$ and a.e. $x,y\in \R^d$, $\pi_i \mathcal{K}_{ij}(x-y) = \pi_j \mathcal{K}_{ji}(x-y)$.  In \cite{CSS}, the authors assume the interaction kernels are bounded and integrable, and they show that for sufficiently small initial mass, weak solutions to (\ref{PDE}) exist globally in time.  They go on to show that, assuming the kernels are of bounded variation and satisfy the detailed balance condition, global existence of weak solutions holds for arbitrary initial mass.  Finally, they show global existence of non-negative classical solutions under the additional assumptions that the interaction kernels are compactly supported and the initial data is smooth.  Under a smallness condition on the initial data, the authors are able to establish global existence of classical solutions assuming only that the kernels belong to $L^1\cap L^{\infty}$ and are of bounded variation, without a compact support or detailed balance assumption.  In another very recent work \cite{CGH}, the authors establish global existence of weak solutions to (\ref{PDE}) in dimensions three and higher for small initial data and for interaction kernels potentially as singular as the Newtonian potential.  In their work, no assumptions are placed on the interaction matrix $H$.       

In this work, we are interested in global existence of classical solutions to (\ref{PDE}) for singular interaction kernels.  We restrict our attention to dimensions $2$ and $3$, as these are the most physical settings.

Short-time existence of smooth solutions to (\ref{PDE}) can be established for smooth initial data and admissible kernels (see Definition \ref{admissible kernel}) by modifying the classical argument for the single-species setting.  For completeness, we provide this modification in the Appendix.  Our objective is to establish conditions on the viscosity, interaction kernels, and initial data which allow us to extend the short-time solution to a global-in-time solution.  We address this in Theorem \ref{smallHexistence} below.  

We note that our result in Theorem \ref{smallHexistence} differs from the existence results described above (with the exception of \cite{JPZ,CGH}) in that we do not require boundedness of the interaction kernels.  We do, however, impose a smallness condition on the initial data.  Meanwhile the result in \cite{JPZ} allows for merely integrable interaction kernels and no smallness condition on the data, but the authors impose a detailed balance condition, which we do not impose.  Finally, while the global existence result in \cite{CGH} allows for more singular kernels than those utilized in this work, the focus in \cite{CGH} is on weak solutions in dimensions three and higher, while we are interested in global smooth solutions in dimensions two and three.     

In addition to proving Theorem \ref{smallHexistence} below, we include a proof of a single-species analogue, given in Theorem \ref{smallHexistencesingle}. Theorem \ref{smallHexistencesingle} and small variants are essentially known \cite{Bedrossian, BRB}, but we include the proof to highlight the differences between the single and multi-species settings.     

We now state our main existence result.  Below, and in what follows,
$$ \nu_{min} = \min_{1\leq i \leq N} \nu_i,$$ and $$ \| H \|= \max_{1\leq i,j \leq N} |h_{ij}|.$$
Acceptable and allowable kernels are introduced in Definitions \ref{acceptable} and \ref{allowable}, respectively.
\begin{theorem}\label{smallHexistence}
Let $k\geq 3$, $r>2$, and $N\geq 2$.  Assume that $\rho$ is a non-negative solution to (\ref{PDE}) on $[0,T]$ belonging to $Lip([0,T]; W^{k,1}\cap W^{k,\infty}(\R^d))$.  Then there exists a constant $C>0$, depending on $N$, $r$, and when $d=2$, the support of $\mathcal{K}$, such that if either $\mathcal{K}$ is acceptable and $\nu$ satisfies   
\begin{align}\label{viscosityassumptionmain}
\nu_{min} >
\begin{cases}
C \| H\|(\|\mathcal{K}\|_{L^{r}}+ \|\nabla\mathcal{K}\|_{L^2})\|\rho_0\|_{W^{1,1}},  &d=2,\\
C \| H\|(\|\mathcal{K}\|_{L^{3}}+ \|\nabla\mathcal{K}\|_{L^2})(\|\rho_0\|_{L^1\cap L^2}+ \|\nabla\rho_0\|_{L^1}),  &d=3,
\end{cases}
\end{align}
or $\mathcal{K}$ is allowable and $\nu$ satisfies 
\begin{align}\label{viscosityassumptionmain1}
\nu_{min} >
\begin{cases}
C \| H\|\|\nabla\mathcal{K}\|_{L^{r}}\|\rho_0\|_{L^1},  &d=2,\\
C\| H\|\|\nabla\mathcal{K}\|_{L^{3}}\|\rho_0\|_{L^1},  &d=3,
\end{cases}
\end{align}
then $\rho$ can be extended to a unique non-negative global-in-time smooth solution to (\ref{PDE}) satisfying $\rho \in Lip([0,\infty);W^{k,1} \cap W^{k,\infty}(\R^d))$.
\end{theorem}
As mentioned above, short-time existence of a unique solution to (\ref{PDE}) under the assumptions of Theorem \ref{smallHexistence} is a straightforward modification of the case $N=1$, and we provide a proof in the Appendix.  

The main idea in the proof of Theorem \ref{smallHexistence} is to show that under the stated assumptions, $\|\rho(t) \|_{L^2(\R^d)}$ decays with time.  Such decay follows from classical arguments when the divergence of the velocity is non-negative; we give a refined argument similar to that of \cite{Bedrossian} which utilizes our assumptions on the viscosity.  The case $N\geq 2$ presents challenges that are not present when $N=1$, since the velocity of a given species will now depend on the other species in the model.  

The $L^2$-decay estimate, combined with short-time existence as given by Theorem \ref{small data result} and a bootstrapping argument, yields global existence of smooth solutions.  As one might expect from the assumption (\ref{viscosityassumptionmain}), our bootstrapping argument relies heavily on control of the $L^1$-norm of the density gradient.  This control is addressed in Lemma \ref{gradient bounds lemma}.     

The second objective of this work, addressed in Theorems \ref{comparison1} and \ref{comparison2} below, is to establish conditions on the interaction kernel $\mathcal{K}$ and the interaction matrix $H$ which ensure that pointwise inequalities comparing species densities are preserved by the evolution.  

Before stating Theorems \ref{comparison1} and \ref{comparison2}, we remark that the assumptions of these two theorems imply the assumptions of Theorem \ref{smallHexistence} in the case of allowable $\mathcal{K}$.  In particular, the ideal kernels in Theorems \ref{comparison1} and \ref{comparison2} below are also allowable (see Definition \ref{ideal}), and the assumptions on $\nu$ in Theorems \ref{comparison1} and \ref{comparison2} imply the assumption on $\nu$ in (\ref{viscosityassumptionmain1}) above.  Thus, Theorem \ref{smallHexistence} allows us to consider global-in-time solutions to (\ref{PDE}) in Theorems \ref{comparison1} and \ref{comparison2}.

\begin{theorem}\label{comparison1}
\label{Positivity lemma for divergence}
Let $k\geq 3$, $r>2$, and $N\geq 2$.  Assume that $\mathcal{K}$ is ideal, and that for each $i$ between $1$ and $N$, $\mathcal{K}_i= \mathcal{K}_{i1}=\mathcal{K}_{i2}=\dots =\mathcal{K}_{iN}$.  Let $C_{\mathcal{K}}$ be as in Definition \ref{ideal}.  Also assume that $\nu_i = \nu$ for every $1\leq i\leq N$, and that $\rho \in Lip([0,\infty); W^{k,1} \cap W^{k,\infty}(\R^d))$ is a non-negative, global-in-time solution to (\ref{PDE}).  Finally, assume that there exists a column vector of real numbers $\gamma=[\gamma_1, \dots, \gamma_N]^T$ such that, if $\underline{h} = [h_{11}, h_{12},\dots, h_{1N}]$, then 
\begin{equation*}
H=\gamma \underline{h}.
\end{equation*}
For $1\leq i \leq N$, set $$\theta_i = h_{1i}\rho_i$$ and $$\theta(x,t) \eqdef \sum_{i=1}^N \theta_i(x,t).$$ 

Under the above assumptions, there exists a constant $C>0$, depending on $r$ and $N$, such that if $\theta_0(x)\geq 0$ on $\R^d$ and 
\begin{align}\label{viscosityassumptioncomparison1}
\nu >
\begin{cases}
C \| H\|(\|\mathcal\nabla {K}\|_{L^{r}}+ \|\Delta\mathcal{K}\|_{L^2}+C_{\mathcal{K}})\|\rho_0\|_{L^1\cap W^{1,\infty}},  &d=2,\\
C \| H\|(\|\mathcal\nabla{K}\|_{L^2\cap L^{3}}+ \|\Delta\mathcal{K}\|_{L^2}+ C_{\mathcal{K}}) \|\rho_0\|_{L^1\cap W^{1,\infty}},  &d=3,
\end{cases}
\end{align}
then $\theta(x,t) \geq 0$ on $\R^d \times [0,\infty)$.
\end{theorem} 
The significance of Theorem \ref{comparison1} can perhaps be best understood when $N=2$.  If, say, $\theta(x,0) = h_{11} \rho_1(x,0) + h_{12}\rho_2(x,0)\geq 0$ for some $h_{11}$, $h_{12}$ and for every $x\in\R^d$, then $h_{11} \rho_1(x,0) \geq  -h_{12}\rho_2(x,0)$.  Theorem \ref{comparison1} states that, if the interaction kernels are sufficiently localized (see Definition \ref{ideal}), then the pointwise comparison between the two species is preserved over time; specifically, $h_{11} \rho_1(x,t) \geq  -h_{12}\rho_2(x,t)$ on $\R^d\times [0,\infty)$.  

Note also that for $1\leq j\leq N$, under the additional assumption that $\gamma_j\mathcal{K}_j\geq 0$ a.e., Theorem \ref{comparison1} states that if the linear combination of sensing functions for species $j$, given by 
\begin{equation*}\label{velacrossj}
\gamma_j\mathcal{K}_j\ast \theta,
\end{equation*}
is initially non-negative, then it must remain non-negative as time evolves.

Theorem \ref{comparison2} below is similar to Theorem \ref{comparison1}, but is specific to $N=2$ and allows for a different set of assumptions on $H$.  Specifically, we do not require that $H$ has rank $1$.     

\begin{theorem}\label{comparison2} Let $k\geq 3$, $N = 2$, $r>2$, and $\nu_1=\nu_2=\nu$.  Assume $\rho_1$ and $\rho_2$ are non-negative, global-in-time solutions to (\ref{PDE}) in $Lip([0,\infty);W^{k, 1}\cap W^{k,\infty}(\R^d))$.  Further assume that $\mathcal{K}$ is ideal (with $C_{\mathcal{K}}$ as in Definition \ref{ideal}) and that there exists $c_0>0$ such that for every $x\in\R^d$, $$c_0( h_{11}\mathcal{K}_{11} -h_{21} \mathcal{K}_{21})(x) = (h_{22}\mathcal{K}_{22}-h_{12}\mathcal{K}_{12})(x)$$ and $$ \rho_1(x,0) \geq c_0\rho_2(x,0).$$ Then there exists a constant $C>0$, depending on $r$ and $N$, such that if
\begin{align}\label{viscosityassumptioncomparison2}
\nu >
\begin{cases}
C \| H\|(\|\mathcal\nabla {K}\|_{L^{r}}+ \|\Delta\mathcal{K}\|_{L^2}+C_{\mathcal{K}})\|\rho_0\|_{L^1\cap W^{1,\infty}},  &d=2,\\
C \| H\|(\|\mathcal\nabla{K}\|_{L^2\cap L^{3}}+ \|\Delta\mathcal{K}\|_{L^2}+ C_{\mathcal{K}})\|\rho_0\|_{L^1\cap W^{1,\infty}},  &d=3,
\end{cases}
\end{align}
then $\rho_1(x,t) \geq c_0\rho_2(x,t)$ for all $(x,t)\in \R^d\times [0,\infty)$.
\end{theorem}
To understand Theorem \ref{comparison2}, observe that if we also assume $$h_{11}\mathcal{K}_{11} -h_{21} \mathcal{K}_{21}\geq 0$$ on $\R^d$, then the conditions on $H$ imply that
\begin{equation*}
\begin{split}
&c_0( h_{11}\mathcal{K}_{11} -h_{21} \mathcal{K}_{21})\ast\rho_1(x,0) = (h_{22}\mathcal{K}_{22}-h_{12}\mathcal{K}_{12})\ast\rho_1(x,0)\\
&\qquad \geq c_0(h_{22}\mathcal{K}_{22}-h_{12}\mathcal{K}_{12})\ast \rho_2(x,0).
\end{split}
\end{equation*}
Dividing through by $c_0$ and rearranging terms, this gives
$$h_{11}\mathcal{K}_{11} \ast  \rho_1(x,0) + h_{12}\mathcal{K}_{12} \rho_2(x,0) \geq h_{21}\mathcal{K}_{21} \ast \rho_1(x,0) + h_{22}\mathcal{K}_{22}\ast \rho_2(x,0).$$
Theorem \ref{comparison2} states that this property is preserved over time; that is, the linear combination of sensing functions for species $1$ will remain larger over time than that for species $2$.

We note that while the techniques used to prove Theorem \ref{comparison2} are similar to those used to prove Theorem \ref{comparison1}, Theorem \ref{comparison2} is not a consequence of Theorem \ref{comparison1}.  To see this, consider the example $\mathcal{K}_{ij} = \mathcal{K}_{11}$ for all $i,j$ and \[    H=   \left[ {\begin{array}{cc}    10 & -9 \\    9 & -8 \\   \end{array} } \right], \]
with $c_0 = 1$, which satisfies the conditions of Theorem \ref{comparison2}, but not those of Theorem \ref{comparison1}.

The paper is organized as follows.  In Section \ref{Definitions and Preliminary Lemmas}, we state some useful definitions and lemmas.  In Section \ref{smallH}, we prove Theorems \ref{smallHexistencesingle} and \ref{smallHexistence}.  In Section \ref{comparisons}, we prove Theorems \ref{comparison1} and \ref{comparison2}. Finally, in the Appendix, we prove several a priori estimates as well as short-time existence of smooth solutions to (\ref{PDE}).

\section{Definitions and Preliminary Lemmas} \label{Definitions and Preliminary Lemmas}  

In this section, we state some definitions and lemmas that will be useful in subsequent sections.  
\subsection{Notational Conventions and Definitions} We begin with some notational conventions and definitions.
\begin{definition}\label{vector norms}
Let $p \in [1,\infty]$. If $f=(f_1,\hdots, f_N)$ is a vector-valued function, we say $f$ belongs to $L^p(\R^d)$ if $f_i\in L^p(\R^d)$ for all $i$, $1\leq i \leq N$.  We set  
$$\| f \|_{L^p}
=
\max_{1\leq i \leq N} \norm{f_i}{L^p(\R^d)} <\infty.
$$
Similarly, if $f=(f_{ij})_{1\leq i,j\leq N}$ is matrix-valued, we say $f$ is in $L^p(\R^d)$ if $f_{ij}\in L^p(\R^d)$ for all $i,j$, $1\leq i,j \leq N$.  In this case, we set  
$$\| f \|_{L^p}
=
\max_{1\leq i,j \leq N} \norm{f_{ij}}{L^p(\R^d)} <\infty.
$$
\end{definition} 

\begin{definition}(Bochner Space) Let $(A,\mathcal{S},\mu)$ be a measure space, $(X,\norm{\cdot}{X})$ be a Banach space, and $1 \leq p \leq \infty$.  The Bochner space $L^p(A;X)$ is defined via the norm
\begin{equation}
\begin{cases}
\norm{f}{L^p(A;X)} = \Large \left(\int_{A} \norm{f(t)}{X}^p d\mu(t)\right)^{1/p}, 1\leq p <\infty,\\
\norm{f}{L^{\infty}(A;X)} = \sup_{t \in A} \norm{f(t)}{X} , \text{  }p=\infty.\\
\end{cases}
\end{equation}
\end{definition}

In what follows, for $p,q\in[1,\infty]$, we equip the Banach space $L^p\cap L^q(\R^d)$ with the norm
\begin{equation*}
\| f\|_{L^p\cap L^q} = \| f \|_{L^p} + \| f \|_{L^q}.
\end{equation*}

Throughout the paper, we let $C$ denote a constant which may change from line to line.  We will make explicit the dependence of constants on various quantities, such as $\nu$ and norms of functions, when this dependence is significant to the results.  In general, to simplify the presentation, we will not make explicit the dependence of constants on $N$ and $d$. 

\Obsolete{We let $a:\R^d \to \R$ denote a radially symmetric bump function such that $a \in C_0^{\infty}(\R^d)$ with $\text{supp } a = \{x \in \R^d: |x| \leq 2\}$.  In addition, we assume $a$ is identically $1$ in $B_1(0)$ and is monotone decreasing for $1 \leq |x| \leq 2$.}

\subsection{Interaction Kernels} We now define the various types of interaction kernels used throughout this work.  We begin with the set of admissible kernels.  In Appendix \ref{A Priori Estimates}, we will assume the interaciton kernel is admissible and establish short-time existence of smooth solutions to (\ref{PDE}).  We note that short-time existence can be established (for both $N=1$ and $N\geq 2$) under weaker assumptions on $\mathcal{K}$, but the additional assumptions will simplify the arguments and are more consistent with the assumptions placed on $\mathcal{K}$ throughout the rest of the paper.  
\begin{definition}
\label{admissible kernel}
(Admissible Kernels) We say that $\mathcal{K}$ is admissible if for every $i,j$ with $1\leq i,j\leq N$,
\begin{enumerate}
\item $\mathcal{K}_{ij}\in L_{loc}^{1}(\R^d)$,
\item
$\mathcal{K}_{ij}(x) = k_{ij}(\abs{x}) = k_{ij}(r)$ is radially symmetric,
\item $\nabla \mathcal{K}_{ij} \in L^2(\R^d)$. 
\end{enumerate}
\end{definition}
\noindent We denote the set of admissible kernels by $\mathcal{A}$.


\begin{definition}(Acceptable Kernels)\label{acceptable}
We say $\mathcal{K} \in \mathcal{A}$ is acceptable if for every $i,j$ with $1\leq i,j \leq N$, 
\begin{enumerate}
\item 
($d=2$ only) $\mathcal{K}_{ij}\in L^{r}(\R^2)$ for some $r>2$ and is compactly supported,   
\item 
($d=3$ only) $\mathcal{K}_{ij}\in L^{3}(\R^3)$.
\end{enumerate}
\end{definition}

In Section \ref{smallH}, we also make use of allowable kernels.

\begin{definition}(Allowable Kernels)\label{allowable}
We say $\mathcal{K} \in \mathcal{A}$ is allowable if for every $i,j$ with $1\leq i,j \leq N$, 
\begin{enumerate}
\item 
($d=2$ only) $\nabla\mathcal{K}_{ij}\in L^{r}(\R^2)$ for some $r>2$ and is compactly supported,  
\item 
($d=3$ only) $\nabla\mathcal{K}_{ij}\in L^{3}(\R^3)$.
\end{enumerate}
\end{definition}

Finally, in Section \ref{comparisons}, we make use of ideal kernels, defined as follows.

\begin{definition}(Ideal Kernels)\label{ideal}
We say $\mathcal{K} \in \mathcal{A}$ is ideal if $\mathcal{K}$ is allowable, and if for every $i,j$ with $1\leq i,j \leq N$, $\Delta \mathcal{K}_{ij}$ belongs to $L^2(\R^d)$ and there exists $C_{ij}>0$ such that 
\begin{equation*} 
\abs{\nabla \mathcal{K}_{ij}(x)}, \abs{\Delta \mathcal{K}_{ij}(x)} \leq
\begin{cases} 
\frac{C_{ij}}{|x|^{d}} \text{ for a.e. } x \in B_1(0)\\
\frac{C_{ij}}{|x|^{d+2}} \text{ for a.e. } x \in \R^d\backslash B_1(0).
\end{cases} 
\end{equation*} 
For an ideal kernel $\mathcal{K}$, we set $C_{\mathcal{K}}=\max_{1\leq i,j\leq N} C_{ij}$.
\end{definition}
\begin{remark}
The membership of $\mathcal{K}$ to $W^{1,1}_{loc}(\R^d)$ and Definition \ref{ideal} imply that ideal kernels satisfy $\nabla \mathcal{K}\in L^1(\R^d)$ and $\Delta \mathcal{K}\in L^1(\R^d)$. 
\end{remark}
\subsection{Useful Inequalities}
In what follows, we make use of the following well-known Sobolev inequality, which is proved in \cite{AB}, Lemma 7.17.
\begin{lemma}\label{Sobolev inequality}
Assume $p\in [1,d)$ and q = dp/(d-p).  Then there exists $C>0$ depending on $d$ and $p$ such that, for all $f\in W^{1,p}(\R^d)$, $$\|f\|_{L^q} \leq C \|\nabla f \|_{L^p}.$$
\end{lemma}
We also make use of the following interpolation inequality, due to Nash \cite{Nash}.
\begin{lemma}\label{Nash inequality}
There exists $C>0$ depending on $d$ such that for all $f\in L^1\cap W^{1,2}(\R^d)$, $$\| f \|^{1+2/d}_{L^2} \leq C \| f \|^{2/d}_{L^1}\|\nabla f \|_{L^2}.$$
\end{lemma}   
\subsection{Mild solutions of (\ref{aggregation equation}) and (\ref{PDE})} We now introduce the semigroup operator for the Laplacian and some of its properties. This operator arises in the definition of a mild solution of (\ref{aggregation equation}) and (\ref{PDE}). 
\begin{definition}
Define the semigroup operator for the Laplacian $e^{\nu\Delta t}$ via its Fourier transform as

$$
\mathcal{F}(e^{\nu\Delta t}f(\cdot)) = \mathcal{F}(g(\cdot,t)) \mathcal{F}(f(\cdot)),
$$
where
$$
\mathcal{F}(g)(\xi,t) = e^{-\nu |\xi|^{2}t}.
$$
\end{definition}
\begin{proposition}
\label{fractional heat kernel bound}
Let $p\in [1,\infty]$, $t>0$ and $f \in L^1 \cap L^{p}(\R^d)$.  Then $$\norm{e^{\nu\Delta t}f}{L^1 \cap L^{p}} \leq \norm{f}{L^1 \cap L^{p}}.$$
\end{proposition}
\begin{proof}
The result immediately follows from an application of Young's convolution inequality and the fact that 
$$
\norm{g(\cdot,t)}{L^1}=\mathcal{F}(g)(0,t)=1.
$$
\end{proof}

We refer to \cite{Wu} for a proof of the next proposition.  Note that the proof in \cite{Wu} is specific to $d=2$ but can easily be extended to $d \geq 2$.
\begin{proposition}
\label{semigroup bounds}
Let $1\leq p \leq q \leq \infty$ and $t>0$.  Then the operators $e^{\nu\Delta t}$ and $\nabla e^{\nu\Delta t}$ are bounded from $L^p(\R^d)$ to $L^q(\R^d)$. Specifically, we have 
$$
\norm{e^{\nu\Delta t}f}{L^q} \leq C(\nu t)^{-\frac{d}{2}\left(\frac{1}{p}-\frac{1}{q}\right)}\norm{f}{L^p},
$$
$$
\norm{\nabla e^{\nu\Delta t}f}{L^q} \leq C(\nu t)^{-\left(\frac{1}{2}+\frac{d}{2}\left(\frac{1}{p}-\frac{1}{q}\right)\right)}\norm{f}{L^p},
$$
where $C$ is a constant depending on $d$, $p$, and $q$ only. 
\end{proposition}
 
We now define a mild solution of (\ref{aggregation equation}) when $N=1$, or, analogously, (\ref{PDE}) when $N\geq 2$.
\begin{definition}\label{Mild Solution Def}
We say that $\rho$ is a mild solution of (\ref{aggregation equation}) when $N=1$, or \eqref{PDE} when $N\geq 2$, if $\rho$ belongs to $L^{\infty}([0,T];L^1 \cap L^{\infty}(\R^d))$ and satisfies
\begin{equation}
\label{mild solution}
\rho(\cdot,t)=e^{\nu\Delta t}\rho_0 - \int_0^t \nabla e^{\nu\Delta(t-\tau)}(u\rho)(\tau)\, d\tau \hspace{.5cm} \text{ for all } t \in [0,T]. 
\end{equation}
Here $u=\nabla\mathcal{K}\ast\rho$ when $N=1$ and $u=\nabla((H\circ\mathcal{K})\ast \rho)$ when $N\geq 2$, with $\mathcal{K}$ admissible.
\end{definition}

\section{Proof of Theorem \ref{smallHexistence}}\label{smallH}

In this section, we prove Theorem \ref{smallHexistence}.  We also prove a single-species analogue, which we now state.

\begin{theorem}\label{smallHexistencesingle}
Let $T>0$ be as in Theorem \ref{small data result}.  Let $k\geq 3$ and $N=1$, and assume that $\mathcal{K}\in
\mathcal{A}\cap L^1(\R^d)$.  Assume $\rho$ is a non-negative solution to (\ref{aggregation equation}) on $[0,T]$ belonging to $Lip([0,T]; W^{k,1}\cap W^{k,\infty}(\R^d))$.  Then there exists an absolute constant $C>0$ such that, if $\nu$ satisfies  
\begin{equation}\label{viscosityassumptionmainsingle}
\nu >
C (\|\mathcal{K}\|_{L^{1}}+ \|\nabla\mathcal{K}\|_{L^2})\|\rho_0\|_{L^1\cap L^{\infty}}\\
\end{equation}
then $\rho$ can be extended to a unique non-negative global-in-time smooth solution to (\ref{aggregation equation}) satisfying $\rho \in Lip([0,\infty);W^{k,1} \cap W^{k,\infty}(\R^d))$.
\end{theorem}
Before proving the theorems, we establish bounds on several key quantities.
\\

\noindent {\bf Bounds on $\|\rho\|_{L^2}$, $\|\rho\|_{L^{\infty}}$, and $\|\nabla\rho\|_{L^{1}}$.} The following two lemmas, showing decay of the $L^2$-norm of the density, are the key ingredients in the proofs of Theorems \ref{smallHexistencesingle} and \ref{smallHexistence}, respectively.
\begin{lemma}\label{energydecreasesingle}
Let $T$ be as in Theorem \ref{small data result}, let $k\geq 3$, and let $N=1$.  Assume $\mathcal{K}\in\mathcal{A}\cap L^1(\R^d)$ and $\rho$ is a solution to (\ref{aggregation equation}) in $Lip([0,T];W^{k,1} \cap W^{k,\infty}(\R^d))$.  Under these assumptions, if 
\begin{equation}\label{viscosityassumptionsingle}
\nu > 2\|\mathcal{K}\|_{L^{1}}\|\rho_0\|_{L^1 \cap L^{\infty}},\\
\end{equation}
then for all $t\in [0,T]$,
\begin{equation}\label{L2decaysingle}
\|\rho(t)\|_{L^2} \leq \frac{\|\rho_{0}\|_{L^2}}{\left(1+ C_1\eta t\|\rho_{0}\|^{4/d}_{L^2}\right)^{d/4}}.
\end{equation}
Here $C_1=C\|\rho_0\|^{-4/d}_{L^1}$, and
$$\eta = \nu - 2\|\mathcal{K}\|_{L^1}\|\rho_0\|_{L^1\cap L^{\infty}}.$$
\end{lemma}
\begin{proof}
We first multiply (\ref{aggregation equation}) by $\rho$ and integrate in the spatial variable, which gives
\begin{equation}\label{energyibpsingle}
\begin{split}
&\frac{d}{dt} \|\rho(t)\|^2_{L^2} = 2\int_{\R^d} (\nu\rho\Delta\rho - \rho\nabla\cdot(u\rho) ) \, dx.\\
\end{split}
\end{equation} 
We manipulate each term on the right hand side, beginning with $2\int_{\R^d} \nu\rho\Delta\rho \, dx$.  Integration by parts yields 
\begin{equation}\label{visctermsingle}
\begin{split}
&2\nu\int_{\R^d} \rho\Delta\rho \, dx = -2\nu  \int_{\R^d} |\nabla\rho|^2 \, dx.
\end{split}
\end{equation} 
We now examine $2\int_{\R^d} \rho\nabla\cdot(u\rho)  \, dx$.  Applying the product rule and again integrating by parts, we have
\begin{equation*}
\begin{split}
&2\int_{\R^d} \rho\nabla\cdot(u\rho)  \, dx = 2\int_{\R^d} \rho^{2}\nabla\cdot u  \, dx + 2\int_{\R^d} \rho u \cdot \nabla\rho \, dx \\
&\qquad = 2\int_{\R^d} \rho^{2}\nabla\cdot u  \, dx-2\int_{\R^d} \rho\nabla\cdot(u\rho)  \, dx.
\end{split}
\end{equation*}
Thus,
\begin{equation}\label{PRIBP}
\begin{split}
&2\int_{\R^d} \rho\nabla\cdot(u\rho)  \, dx = \int_{\R^d} \rho^{2}\nabla\cdot u  \, dx = \int_{\R^d} \rho^{2}\left(\Delta\mathcal{K}\ast \rho \right) \, dx\\
& \qquad= -\int_{\R^d} \nabla(\rho^{2})\cdot\left(\mathcal{K}\ast \nabla\rho  \right)\, dx = -2\int_{\R^d} \rho\nabla\rho\cdot\left(\mathcal{K}\ast \nabla\rho \right) \, dx.
\end{split}
\end{equation} 
By (\ref{PRIBP}), H\"older's inequality, and Young's inequality,
\begin{equation}\label{term2boundsingle}
\begin{split}
& 2\int_{\R^d} \rho\nabla\cdot(u\rho)  \, dx  \leq 2\|\rho\|_{L^{\infty}} \| \nabla\rho \|_{L^2} \| \mathcal{K}\|_{L^1}\|\nabla \rho \|_{L^2}\\
&\qquad = 2\|\rho\|_{L^{\infty}}  \| \mathcal{K}\|_{L^1}\| \nabla\rho \|^2_{L^2} \leq 4\|\rho_0\|_{L^1\cap L^{\infty}}  \| \mathcal{K}\|_{L^1}\| \nabla\rho \|^2_{L^2},
\end{split}
\end{equation} 
where we applied the bound $\|\rho\|_{L^{\infty}([0,T]; L^{\infty})}\leq 2\|\rho_0\|_{L^1\cap L^{\infty}}$ from Theorem \ref{Mild Solution Proof} with $p=\infty$ to get the last inequality.  Substituting (\ref{term2boundsingle}) and (\ref{visctermsingle}) into (\ref{energyibpsingle}) gives 
\begin{equation*}
\begin{split}
&\frac{d}{dt} \|\rho(t)\|^2_{L^2} \leq - 2\nu\| \nabla\rho \|^2_{L^2} + 4\|\rho_0\|_{L^1\cap L^{\infty}}  \| \mathcal{K}\|_{L^1}\| \nabla\rho \|^2_{L^2}\\
&\qquad = - 2\left(\nu -  2\|\rho_0\|_{L^1\cap L^{\infty}}\| \mathcal{K} \|_{L^{1}} \right)\| \nabla\rho \|^2_{L^2}.
\end{split}
\end{equation*}
We conclude that on $[0,T]$,
\begin{equation*}
\begin{split}
&\frac{d}{dt} \|\rho(t)\|^2_{L^2} \leq  -2\eta\| \nabla\rho \|^2_{L^2},
\end{split}
\end{equation*} 
where 
$$\eta =\nu -  2\|\rho_0\|_{L^1\cap L^{\infty}}\| \mathcal{K} \|_{L^{1}}.$$

To establish (\ref{L2decaysingle}) and complete the proof of Lemma \ref{energydecreasesingle}, we apply Lemma \ref{Nash inequality} to conclude that
\begin{equation*}
\begin{split}
&\| \rho \|^{2(1+2/d)}_{L^2}\leq C\| \rho \|^{4/d}_{L^1} \| \nabla \rho \|^2_{L^2} ,\\
\end{split}
\end{equation*}
so that, by conservation of mass,
\begin{equation*}
\begin{split}
&\| \nabla \rho \|^2_{L^2} \geq C\|\rho_0\|^{-4/d}_{L^1}\| \rho \|_{L^2}^{2(1+2/d)}.\\
\end{split}
\end{equation*}
We conclude that
$$ \frac{d}{dt} \|\rho(t)\|^2_{L^2} \leq  - C_1\eta(\| \rho \|_{L^2}^{2})^{\left(1+\frac{2}{d}\right)},$$
where $C_1=C\|\rho_0\|^{-4/d}_{L^1}$.  It follows that 
\begin{equation}
\|\rho(t)\|^2_{L^2} \leq \frac{\|\rho_{0}\|^2_{L^2}}{\left(1+ C_1\eta t\|\rho_{0}\|^{4/d}_{L^2}\right)^{d/2}}.
\end{equation}
\end{proof}
We now establish an $L^2$-decay estimate in the multi-species setting.  Our strategy will be similar to that used in the single-species case; however, the multi-species setting presents some difficulties, as now the divergence of the velocity of a given species depends on densities of the other species.  This complication necessitates stronger assumptions on the interaction kernel, as we see in Lemma \ref{energydecrease} below.
\begin{lemma}\label{energydecrease}
Let $T$ be as in Theorem \ref{small data result}, $k\geq 3$, $r>2$, and $N\geq 2$.  Let $\rho$ be a solution to (\ref{PDE}) in $Lip([0,T];W^{k,1} \cap W^{k,\infty}(\R^d))$.  Under these assumptions, there exists a constant $C>0$ depending on $d$, $N$, $r$, and, when $d=2$, the support of $\mathcal{K}$ or $\nabla\mathcal{K}$, such that, if either\\
\\
\textbf{Case 1:} $\mathcal{K}$ is acceptable and $\nu$ satisfies
\begin{align}\label{viscosityassumption}
\nu_{min} >
\begin{cases}
C \| H\|\|\mathcal{K}\|_{L^{r}}\|\nabla\rho_0\|_{L^1},  &d=2,\\
C\| H\|\|\mathcal{K}\|_{L^{3}}\|\nabla\rho_0\|_{L^1},  &d=3,
\end{cases}
\end{align}
or\\
\\
\textbf{Case 2:} $\mathcal{K}$ is allowable and $\nu$ satisfies
\begin{align}\label{viscosityassumption1}
\nu_{min} >
\begin{cases}
C \| H\|\|\nabla\mathcal{K}\|_{L^{r}}\|\rho_0\|_{L^1},  &d=2,\\
C\| H\|\|\nabla\mathcal{K}\|_{L^{3}}\|\rho_0\|_{L^1},  &d=3,
\end{cases}
\end{align}
then for each $i$ between $1$ and $N$, 
\begin{equation}\label{L2decay}
\|\rho_i(t)\|_{L^2} \leq \frac{\|\rho_{i,0}\|_{L^2}}{\left(1+ C_1\eta_d t\|\rho_{i,0}\|^{4/d}_{L^2}\right)^{d/4}}.
\end{equation}
In (\ref{L2decay}), $C_1=C\|\rho_0\|_{L^1}^{-4/d}$, while in Case 1,
\begin{equation}\label{visccase1}
\eta_d =
\begin{cases}
\nu_{min} -  C\| H \| \| \nabla\rho_0 \|_{L^1} \| \mathcal{K} \|_{L^{r}} >0, &d=2,\\
\nu_{min} -  C\| H \| \| \nabla\rho_0 \|_{L^1} \| \mathcal{K} \|_{L^{3}} >0, &d=3,
\end{cases}
\end{equation}
and in Case 2,
\begin{equation}\label{visccase2}
\eta_d =
\begin{cases}
\nu_{min} -  C\| H \| \| \rho_0 \|_{L^1} \| \nabla\mathcal{K} \|_{L^{r}} >0, &d=2,\\
\nu_{min} -  C\| H \| \| \rho_0 \|_{L^1} \| \nabla\mathcal{K} \|_{L^{3}} >0, &d=3.
\end{cases}
\end{equation}
\end{lemma}
\begin{proof}
Fix $i$ between $1$ and $N$, let $\mathcal{K}$ be acceptable, and assume (\ref{viscosityassumption}) holds.  As in the proof of Lemma \ref{energydecreasesingle}, we multiply (\ref{PDE}) by $\rho_i$ and integrate in the spatial variable, which gives
\begin{equation}\label{energyibp}
\begin{split}
&\frac{d}{dt} \|\rho_i(t)\|^2_{L^2} = 2\int_{\R^d} (\nu_i\rho_i\Delta\rho_i - \rho_i\nabla\cdot(u_i\rho_i) ) \, dx.\\
\end{split}
\end{equation}   
Integrating the diffusion term by parts gives 
\begin{equation}\label{viscterm}
\begin{split}
&2\nu_i\int_{\R^d} \rho_i\Delta\rho_i \, dx = -2\nu_i  \int_{\R^d} |\nabla\rho_i|^2 \, dx.
\end{split}
\end{equation} 
We now examine $2\int_{\R^d} \rho_i\nabla\cdot(u_i\rho_i)  \, dx$.  Again following the proof of Lemma \ref{energydecreasesingle}, we apply the product rule and integrate by parts, giving
\begin{equation}\label{ibpterm2}
\begin{split}
&2\int_{\R^d} \rho_i\nabla\cdot(u_i\rho_i)  \, dx = \int_{\R^d} \rho^{2}_i\nabla\cdot u_i  \, dx = \int_{\R^d} \rho^{2}_i \left(\sum_{j=1}^N h_{ij}\Delta\mathcal{K}_{ij}\ast\rho_j \right) \, dx\\
& \qquad= -\int_{\R^d} \nabla(\rho^{2}_i)\cdot \left(\sum_{j=1}^N h_{ij}\mathcal{K}_{ij}\ast\nabla\rho_j \right) \, dx\\
&\qquad = -2\int_{\R^d} \rho_i\nabla\rho_i\cdot\left(\sum_{j=1}^N h_{ij}\mathcal{K}_{ij}\ast\nabla\rho_j \right) \, dx.
\end{split}
\end{equation} 
We now consider the cases $d=2$ and $d=3$ separately.  We first assume $d=2$.  Let $\delta>0$ be such that supp $\mathcal{K}_{ij} \subseteq B_{\delta}(0)$ for every $i,j$, $1\leq i,j\leq N$.  Fix $r>2$, and assume $1/r + 1/s =1$.  Expanding the convolution in (\ref{ibpterm2}) and utilizing the compact support of $\mathcal{K}_{ij}$ and Fubini's Theorem, we find that
\begin{equation*}
\begin{split}
& \left | 2\int_{\R^2} \rho_i\nabla\cdot(u_i\rho_i)  \, dx \right | \\
&= \left |2\int_{\R^2}\sum_{j=1}^N \left(h_{ij}\nabla\rho_j(y) \left(\int_{B_{\delta}(y)} \rho_i(x)\nabla\rho_i (x) \mathcal{K}_{ij}(x-y) \, dx \right) \right) \, dy \right |\\
&\qquad \leq C\| H \| \| \nabla\rho \|_{L^1} \| \mathcal{K} \|_{L^{r}} \sup_{y\in\R^2} \| \rho_i \cdot \nabla\rho_i \|_{L^s(B_{\delta}(y))}\\
&\qquad \leq C\| H \| \| \nabla\rho \|_{L^1} \| \mathcal{K} \|_{L^{r}} \sup_{y\in\R^2} \left(\| \rho_i \|_{L^q(B_{\delta}(y))} \|\nabla\rho_i \|_{L^2(B_{\delta}(y))} \right),
\end{split}
\end{equation*} 
where $1/2 + 1/q = 1/s$.  Above we used H\"older's inequality and translation invariance of the $L^r$-norm.  Now observe that by Lemma \ref{Sobolev inequality} and H\"older's inequality,
\begin{equation*}
\| \rho_i \|_{L^q(B_{\delta}(y))} \leq \| \nabla \rho_i \|_{L^{\frac{2q}{2+q}}(B_{\delta}(y))} \leq (m(B_{\delta}(y)))^{\frac{2}{2+q}}\| \nabla \rho_i \|_{L^2(B_{\delta}(y))}.
\end{equation*}
We conclude that
\begin{equation}\label{term2bound}
2\int_{\R^2} \rho_i\nabla\cdot(u_i\rho_i)  \, dx \leq C\delta^{\frac{4}{2+q}}\| H \| \| \nabla\rho \|_{L^1} \| \mathcal{K} \|_{L^{r}} \| \nabla\rho_i \|^2_{L^2},
\end{equation}
where $C$ depends on $d$, $N$ and $r$.  Substituting (\ref{term2bound}) and (\ref{viscterm}) into (\ref{energyibp}) gives, for $d=2$, 
\begin{equation*}
\begin{split}
&\frac{d}{dt} \|\rho_i(t)\|^2_{L^2} \leq - 2\nu_i\| \nabla\rho_i \|^2_{L^2} + C\delta^{\frac{4}{2+q}}\| H \| \| \nabla\rho \|_{L^1} \| \mathcal{K} \|_{L^{r}} \| \nabla\rho_i \|^2_{L^2}\\
&\qquad = - 2\left(\nu_i -  C\delta^{\frac{4}{2+q}}\| H \| \| \nabla\rho \|_{L^1} \| \mathcal{K} \|_{L^{r}} \right)\| \nabla\rho_i \|^2_{L^2}.
\end{split}
\end{equation*}
But by (\ref{gradient bound 22}) below,
\begin{equation}\label{gradforvisc1}
\begin{split}
&\| \nabla \rho \|_{L^{\infty}([0,T];L^{1})} \leq C\| \nabla \rho_0 \|_{L^{1}}.
\end{split}
\end{equation}
We conclude that, when $d=2$,
\begin{equation*}
\begin{split}
&\frac{d}{dt} \|\rho_i(t)\|^2_{L^2} \leq  - 2\left(\nu_{min} -  C\delta^{\frac{4}{2+q}}\| H \| \| \nabla\rho_0 \|_{L^1} \| \mathcal{K} \|_{L^{r}} \right)\| \nabla\rho_i \|^2_{L^2},
\end{split}
\end{equation*}
where $C$ depends on $d$, $N$, and $r$.  

We now consider the case $d=3$.  In this case, one again has 
\begin{equation*}
\begin{split}
& \left | 2\int_{\R^3} \rho_i\nabla\cdot(u_i\rho_i)  \, dx \right | \\
&= \left |2\int_{\R^3}\sum_{j=1}^N \left(h_{ij}\nabla\rho_j(y) \left(\int_{\R^3} \rho_i(x)\nabla\rho_i (x) \mathcal{K}_{ij}(x-y) \, dx \right) \right) \, dy \right |\\
&\qquad \leq C\|H\|\| \nabla\rho \|_{L^1} \sup_{1\leq j\leq N}\sup_{y\in\R^3}\left|\int_{\R^3} \rho_i(x)\nabla\rho_i(x) \mathcal{K}_{ij}(x-y) \, dx \right |.
\end{split}
\end{equation*} 
Note that for each $y\in\R^3$ and $1\leq j\leq N$, by translation invariance of the $L^3$-norm and H\"older's inequality,
\begin{equation*}
\begin{split}
&\left|\int_{\R^3} \rho_i(x)\nabla\rho_i(x) \mathcal{K}_{ij}(x-y) \, dx \right | \leq \| \mathcal{K} \|_{L^3} \|\rho_i \nabla\rho_i \|_{L^{3/2}}\\
&\qquad \leq  \| \mathcal{K}\|_{L^3} \|\rho_i \|_{L^6} \|\nabla\rho_i \|_{L^{2}}  \leq  \| \mathcal{K} \|_{L^3} \|\nabla\rho_i \|^2_{L^2} , 
\end{split}
\end{equation*} 
where we applied Lemma \ref{Sobolev inequality} to get the last inequality.  Thus, 
\begin{equation}\label{term2boundd3}
2\int_{\R^2} \rho_i\nabla\cdot(u_i\rho_i)  \, dx \leq C\| H \| \| \nabla\rho \|_{L^1} \| \mathcal{K} \|_{L^{3}} \| \nabla\rho_i \|^2_{L^2},
\end{equation}  
where $C$ depends on $N$.  Substituting (\ref{term2boundd3}) and (\ref{viscterm}) into (\ref{energyibp}) gives, for $d=3$, 
\begin{equation*}
\begin{split}
&\frac{d}{dt} \|\rho_i(t)\|^2_{L^2} \leq - 2\nu_i\| \nabla\rho_i \|^2_{L^2} + C\| H \| \| \nabla\rho \|_{L^1} \| \mathcal{K} \|_{L^{3}} \| \nabla\rho_i \|^2_{L^2}\\
&\qquad = - 2\left(\nu_i -  C\| H \| \| \nabla\rho \|_{L^1} \| \mathcal{K} \|_{L^{3}} \right)\| \nabla\rho_i \|^2_{L^2}.
\end{split}
\end{equation*}
Again by (\ref{gradient bound 22}) below,
\begin{equation*}
\begin{split}
&\| \nabla \rho \|_{L^{\infty}([0,T];L^{1})} \leq  C\| \nabla \rho_0 \|_{L^{1}}.\\
\end{split}
\end{equation*}
We conclude that for $d=2$, $3$, on $[0,T]$,
\begin{equation*}
\begin{split}
&\frac{d}{dt} \|\rho_i(t)\|^2_{L^2} \leq  -2\eta_d\| \nabla\rho_i \|^2_{L^2},
\end{split}
\end{equation*} 
where 
\begin{align}
\eta_d =
\begin{cases}
\nu_{min} -  C\delta^{\frac{4}{2+q}}\| H \| \| \nabla\rho_0 \|_{L^1} \| \mathcal{K} \|_{L^{r}} >0, &d=2,\\
\nu_{min} -  C\| H \| \| \nabla\rho_0 \|_{L^1} \| \mathcal{K} \|_{L^{3}} >0, &d=3,
\end{cases}
\end{align}
and $C$ depends on $d$, $N$ and $r$.

To establish (\ref{L2decay}) and complete the proof of Lemma \ref{energydecrease}, we again apply Lemma \ref{Nash inequality}, exactly as in the proof of Lemma \ref{energydecreasesingle}, to conclude that
$$ \frac{d}{dt} \|\rho_i(t)\|^2_{L^2} \leq  - C_1\eta_d(\| \rho_i \|_{L^2}^{2})^{\left(1+\frac{2}{d}\right)},$$
where $C_1=C\|\rho_{i,0}\|^{-4/d}_{L^1}$.  It follows that
\begin{equation}
\|\rho_i(t)\|^2_{L^2} \leq \frac{\|\rho_{i,0}\|^2_{L^2}}{\left(1+ C_1\eta_d t\|\rho_{i,0}\|^{4/d}_{L^2}\right)^{d/2}}.
\end{equation}

We now show that (\ref{L2decay}) holds under the assumption that $\mathcal{K}$ is allowable and $\nu$ satisfies (\ref{viscosityassumption1}).  We follow a strategy identical to that above, except in place of the equality in (\ref{ibpterm2}), we write
$$2\int_{\R^d} \rho_i\nabla\cdot(u_i\rho_i)  \, dx = -2\int_{\R^d} \rho_i\nabla\rho_i\cdot\left(\sum_{j=1}^N h_{ij}\nabla\mathcal{K}_{ij}\ast\rho_j \right) \, dx.$$
The remainder of the proof is identical to case $1$ above, but with $\nabla\rho_j$ replaced by $\rho_j$ and $\mathcal{K}_{ij}$ replaced by $\nabla\mathcal{K}_{ij}$.  Note also that, rather than bounding $\|\nabla\rho(t)\|_{L^1}$ as in (\ref{gradforvisc1}), we  can now use conversation of mass to control $\|\rho(t)\|_{L^1}$ by $\|\rho_0\|_{L^1}$.  This completes the proof. 
\end{proof}
One could use the same argument as that in the proof of Lemma \ref{energydecreasesingle} or Lemma \ref{energydecrease} to prove an analogous $L^p$ decay estimate for each $p<\infty$.  However, for $\eta$ (respectively, $\eta_d$) to remain non-negative, the argument requires that $\nu$ grow linearly with $p$.  Thus, one cannot pass to the limit as $p$ approaches infinity to establish an $L^{\infty}$ maximum principle.  For this reason, we will instead utilize $L^2$-decay and the mild formulation to prove a bound on the $L^{\infty}$-norm of the density in Lemma \ref{sup bounds lemma} below.  For now, we restrict our attention to the case $N=1$.  We address the case $N\geq 2$ in Lemma \ref{nablarhoinfty}.     
\begin{lemma}\label{sup bounds lemma}
Let $T>0$ be as in Theorem \ref{small data result}.  Assume $N=1$ and $\mathcal{K}\in \mathcal{A}$, and assume that $\rho\in Lip([0,T]; W^{k,1}\cap W^{k,\infty}(\R^d))$ is a solution to (\ref{aggregation equation}) on $[0,T]$.  Then there exists a constant $C>0$, depending only on $N$ and $d$, such that
\begin{equation}\label{sup bound}
\| \rho \|_{L^{\infty}([0,T];L^{\infty})} \leq C\| \rho_0 \|_{L^{\infty}}.
\end{equation}

Now assume that $N =1$ and $\mathcal{K}\in \mathcal{A}\cap L^1(\R^d)$, and let $T>0$ be arbitrary.  If $\rho\in Lip([0,T]; W^{k,1}\cap W^{k,\infty}(\R^d))$ is a solution of (\ref{aggregation equation}) on $[0,T]$, and if (\ref{L2decaysingle}) holds for all $t\in[0,T]$, then there exists a constant $C>0$, depending only on $N$ and $d$, such that
\begin{equation}\label{sup bound 2}
\| \rho \|_{L^{\infty}([0,T];L^{\infty})} \leq \|  \rho_0 \|_{L^{\infty}} + \left(C\eta^{-1/2}\nu^{-1/2}\|\nabla \mathcal{K}\|_{L^2}\|\rho_0\|_{L^{1}}^{\frac{2}{d}}\|\rho_0\|_{L^{2}}^{1-\frac{2}{d}}\right)\| \rho \|_{L^{\infty}([0,T];L^{\infty})},
\end{equation}  
where $\eta$ is as in Lemma \ref{energydecreasesingle}.
\end{lemma}     
\begin{remark}\label{constantlessthan1}We will apply Lemma \ref{sup bounds lemma} in the proof of Theorem \ref{smallHexistencesingle}.  Note that, under the assumptions on $\nu$ in Theorem \ref{smallHexistencesingle}, one can make $$C\eta^{-1/2}\nu^{-1/2}\|\nabla \mathcal{K}\|_{L^2}\|\rho_0\|_{L^{1}}^{\frac{2}{d}}\|\rho_0\|_{L^{2}}^{1-\frac{2}{d}} <1,$$
which then yields the bound $$ \| \rho \|_{L^{\infty}([0,T];L^{\infty})} \leq C \|  \rho_0 \|_{L^{\infty}}$$ for a constant $C > 1$ depending only on $d$ and $N$.  Lemmas \ref{gradient bounds lemma} and \ref{nablarhoinfty} below will be employed similarly, the first in the proof of Theorem \ref{smallHexistence}, and the second in the proofs of Theorems \ref{comparison1} and \ref{comparison2}.   
\end{remark} 
\Obsolete{\begin{remark}
As one might expect, an analogous statement to (\ref{sup bound 2}) holds when $N=1$.  We have
\begin{equation}\label{sup bound 2 single}
\| \rho \|_{L^{\infty}([0,T];L^{\infty})} \leq \|  \rho_0 \|_{L^{\infty}} + \left(C\eta^{-1/2}\nu^{-1/2}\|\nabla \mathcal{K}\|_{L^2}\|\rho_0\|^{2/d}_{L^1}\|\rho_0\|_{L^{2}}^{1-\frac{2}{d}}\right)\| \rho \|_{L^{\infty}([0,T];L^{\infty})}.
\end{equation}
We omit the proof of (\ref{sup bound 2 single}), since it is identical to the proof of (\ref{sup bound 2}), which we give below.  We make use of (\ref{sup bound 2 single}) in the proof of Theorem \ref{smallHexistencesingle}.
\end{remark} }
\begin{proof}
We first prove (\ref{sup bound}).  We take the $L^{\infty}$-norm of the mild formulation (\ref{mild solution}) and apply Proposition \ref{semigroup bounds}.  This gives
\begin{equation}\label{mildinftybound}
\begin{split}
&\|\rho(t) \|_{L^{\infty}} \leq \| \rho_0 \|_{L^{\infty}} +  C\nu^{-1/2}\int_0^t (t-\tau)^{-1/2}\norm{(u\rho)(\tau)}{L^{\infty}}\, d\tau\\
&\qquad \leq \| \rho_0 \|_{L^{\infty}} +  C\nu^{-1/2}\int_0^t (t-\tau)^{-1/2}\norm{u(\tau)}{L^{\infty}}\norm{\rho(\tau)}{L^{\infty}}\, d\tau.
\end{split}
\end{equation}
It follows by Young's inequality that for each $\tau\in [0,t]$,
\begin{equation*}
\begin{split}
&\| u (\tau) \|_{L^{\infty}} \leq  \| \nabla\mathcal{K}\|_{L^2}\| \rho(\tau) \|_{L^{2}}.
\end{split}
\end{equation*}
Substituting this estimate into (\ref{mildinftybound}) and integrating in time gives
\begin{equation*}
\| \rho \|_{L^{\infty}([0,T];L^{\infty})} \leq \| \rho_0 \|_{L^{\infty}} + \left ( C\nu^{-1/2}\| \nabla\mathcal{K}\|_{L^2} T^{1/2} \|\rho\|_{L^{\infty}([0,T]; L^2)}\right)\| \rho \|_{L^{\infty}([0,T];L^{\infty})}.
\end{equation*}
We now apply the bound $\|\rho\|_{L^{\infty}([0,T]; L^2)} \leq 2\|\rho_0\|_{L^1\cap L^{2}}$ from Theorem \ref{Mild Solution Proof} with $p=2$ and the assumption on $T$ in Theorem \ref{Mild Solution Proof}.  We conclude that 
\begin{equation*}
\| \rho \|_{L^{\infty}([0,T];L^{\infty})} \leq \| \rho_0 \|_{L^{\infty}} +C_0\| \rho \|_{L^{\infty}([0,T];L^{\infty})},
\end{equation*}
with $C_0\in (0,1)$.  Subtracting $C_0\| \rho \|_{L^{\infty}([0,T];L^{\infty})}$ from both sides and dividing through by $1-C_0$ gives (\ref{sup bound}).

We now establish (\ref{sup bound 2}).  We follow the proof of (\ref{sup bound}) to get the bound
\begin{equation}\label{grad bound 4}
\begin{split}
&\qquad\| \rho(t) \|_{L^{\infty}} \leq \| \rho_0 \|_{L^{\infty}} + \\
&C\nu^{-1/2}\int_0^t (t-\tau)^{-1/2}\| \rho (\tau)\|_{L^{\infty}}\|\nabla\mathcal{K}\|_{L^2}\|\rho(\tau)\|_{L^2} \, d\tau.
\end{split}
\end{equation}   
It follows from (\ref{L2decaysingle}) that 
\begin{equation}\label{gradient post Gronwall}
\| \rho \|_{L^{\infty}([0,t]; L^{\infty})} \leq \| \rho_0 \|_{L^{\infty}} + \left (C\nu^{-1/2}\|\nabla\mathcal{K}\|_{L^2}\| \rho_0 \|_{L^{2}}F(t) \right)\| \rho \|_{L^{\infty}([0,t]; L^{\infty})},
\end{equation}   
where \begin{equation}\label{F}F(t) = \int_0^t \frac{(t-\tau)^{-1/2}}{\left(1+C_1\eta\tau \| \rho_0 \|_{L^{2}}^{4/d} \right)^{d/4}} \, d\tau,
\end{equation}
with $C_1$ and $\eta$ as in Lemma \ref{energydecreasesingle}. We claim that $F$ belongs to $L^{\infty}([0,\infty))$ for $d=2$ or $3$.  

First assume $d=2$.  Write
\begin{equation*}
\begin{split}
&F(t) = \int_0^t \frac{(t-\tau)^{-1/2}}{\left(1+C_1\eta\tau \| \rho_0 \|_{L^{2}}^{2} \right)^{1/2}} \, d\tau\\
&= \int_0^{t/2} \frac{(t-\tau)^{-1/2}}{\left(1+C_1\eta\tau \| \rho_0 \|_{L^{2}}^{2} \right)^{1/2}} \, d\tau + \int_{t/2}^t \frac{(t-\tau)^{-1/2}}{\left(1+C_1\eta\tau \| \rho_0 \|_{L^{2}}^{2} \right)^{1/2}} \, d\tau\\
& \leq \frac{(t/2)^{-1/2}}{(C_1\eta)^{1/2}\| \rho_0 \|_{L^{2}}}\int_0^{t/2} \frac{1}{ {\tau}^{1/2}} \, d\tau + \frac{(t/2)^{-1/2}}{(C_1\eta)^{1/2}\| \rho_0 \|_{L^{2}}}\int_{t/2}^t (t-\tau)^{-1/2} \, d\tau\\
& \leq  4(C_1\eta)^{-1/2}\| \rho_0 \|^{-1}_{L^{2}}. 
\end{split}
\end{equation*}  
We conclude that $$ F(t) \leq  4(C_1\eta)^{-1/2}\| \rho_0 \|^{-1}_{L^{2}}.$$
For the case $d=3$, observe that
\begin{equation*}
\begin{split}
&F(t) = \int_0^t \frac{(t-\tau)^{-1/2}}{\left(1+C_1\eta\tau \| \rho_0 \|_{L^{2}}^{4/d} \right)^{d/4}} \, d\tau= \int_0^t \frac{(t-\tau)^{-1/2}}{\left(1+C_1\eta\tau \| \rho_0 \|_{L^{2}}^{4/3} \right)^{3/4}} \, d\tau\\
&\qquad \leq \int_0^t \frac{(t-\tau)^{-1/2}}{\left(1+C_1\eta\tau \left(\| \rho_0 \|_{L^{2}}^{2/3}\right)^{2} \right)^{1/2}} \, d\tau.
\end{split}
\end{equation*}
By an argument very similar to the case $d=2$, we have
$$ F(t) \leq  4(C_1\eta)^{-1/2}\| \rho_0 \|^{-\frac{2}{3}}_{L^{2}}.$$
Substituting the estimates for both $d=2$ and $d=3$ into (\ref{gradient post Gronwall}) gives
$$ \| \rho \|_{L^{\infty}([0,T];L^{\infty})} \leq \| \rho_0 \|_{L^{\infty}} + \left(C(C_1\eta)^{-1/2}\nu^{-1/2}\|\nabla\mathcal{K}\|_{L^2}\|\rho_0\|_{L^{2}}^{1-\frac{2}{d}}\right)\| \rho \|_{L^{\infty}([0,T];L^{\infty})}.$$
Recalling that $C_1=\|\rho_0\|_{L^1}^{-4/d}$, we get the desired estimate.
\end{proof}
We now apply a strategy similar to that above to establish a bound on the $L^1$-norm of the density gradient when $N\geq 2$.  We prove the following lemma.
\begin{lemma}\label{gradient bounds lemma}
Let $T>0$ be as in Theorem \ref{small data result}.  Assume that $\mathcal{K}\in \mathcal{A}$, $N\geq 2$, and $\rho\in Lip([0,T]; W^{k,1}\cap W^{k,\infty}(\R^d))$ is a solution of (\ref{PDE}) on $[0,T]$.  Then there exists a constant $C>0$, depending only on $N$ and $d$, such that $\rho$ satisfies the estimate
\begin{equation}\label{gradient bound 22}
\| \nabla \rho \|_{L^{\infty}([0,T];L^{1})} \leq C\| \nabla \rho_0 \|_{L^{1}}.
\end{equation} 

Now assume that $\mathcal{K}$ is acceptable and $N\geq 2$, and let $T>0$ be arbitrary.  If $\rho\in Lip([0,T]; W^{k,1}\cap W^{k,\infty}(\R^d))$ is a solution of (\ref{PDE}) on $[0,T]$, and if (\ref{L2decay}) holds for all $t\in [0,T]$, then there exists a constant $C>0$ such that
\begin{equation}\label{gradient bound 33}
\begin{split}
&\qquad\qquad\| \nabla \rho \|_{L^{\infty}([0,T];L^{1})} \leq \| \nabla \rho_0 \|_{L^{1}}\\
& + \left(C\eta_d^{-1/2}\nu_{min}^{-1/2}\| H \|\|\nabla\mathcal{K}\|_{L^2}\|\rho_0\|_{L^1}^{2/d}\|\rho_0\|_{L^{2}}^{1-\frac{2}{d}}\right)\| \nabla \rho \|_{L^{\infty}([0,T];L^{1})},
\end{split}
\end{equation}  
where $\eta_d$ is as in (\ref{visccase1}).
\end{lemma}      
\begin{proof}
As in the proof of Lemma \ref{sup bounds lemma}, we utilize the mild formulation, H\"older's inequality, and Proposition \ref{semigroup bounds} to conclude that for each $t\in [0,T]$,
\begin{equation}\label{grad bound 333}
\begin{split}
&\| \nabla\rho(t) \|_{L^{1}} \leq \| \nabla\rho_0 \|_{L^{1}} + \\ C\nu_{min}^{-1/2}
&\int_0^t (t-\tau)^{-1/2}(\norm{\rho(\tau)}{L^{2}}\norm{\nabla \cdot u (\tau)}{L^{2}}+\norm{u(\tau)}{L^{\infty}}\norm{\nabla \rho (\tau)}{L^{1}})\, d\tau.
\end{split}
\end{equation}
By Young's inequality, we have 
\begin{equation*}
\begin{split}
&\| u (\tau) \|_{L^{\infty}} \leq C\| H \| \| \nabla\mathcal{K}\|_{L^2}\| \rho(\tau) \|_{L^{2}},\\
&\norm{\nabla \cdot u(\tau)}{L^{2}} \leq C\| H \| \| \nabla\mathcal{K}\|_{L^2} \norm{\nabla \rho(\tau)}{L^{1}}.
\end{split}
\end{equation*}
Substituting these bounds into (\ref{grad bound 333}) and applying the estimate $\|\rho\|_{L^{\infty}([0,T];L^{2})}\leq 2\| \rho_0 \|_{L^1\cap L^{2}}$ from Theorem \ref{Mild Solution Proof} with $p=2$ yields 
\begin{equation*}
\| \nabla \rho \|_{L^{\infty}([0,T];L^{1})} \leq \| \nabla \rho_0 \|_{L^{1}} + \left ( C\nu_{min}^{-1/2}\| H \|\| \nabla\mathcal{K}\|_{L^2} T^{1/2} \| \rho_0 \|_{L^1\cap L^{2}}\right)\| \nabla \rho \|_{L^{\infty}([0,T];L^{1})}.
\end{equation*}
The estimate (\ref{gradient bound 22}) follows from the assumption on $T$ in Theorem \ref{Mild Solution Proof} and an argument very similar to that used to establish (\ref{sup bound}) in Lemma \ref{sup bounds lemma}.

Finally, we prove (\ref{gradient bound 33}).  We follow the proof of (\ref{gradient bound 22}) to get the bound
\begin{equation}\label{grad bound 44}
\begin{split}
&\qquad\| \nabla\rho(t) \|_{L^{1}} \leq \| \nabla\rho_0 \|_{L^{1}} + \\
&C\| H \|\|\nabla\mathcal{K}\|_{L^2}\nu_{min}^{-1/2}\int_0^t (t-\tau)^{-1/2}\| \rho (\tau)\|_{L^{2}}\|\nabla\rho(\tau) \|_{L^{1}} \, d\tau.
\end{split}
\end{equation}   
It follows from (\ref{L2decay}) that 
\begin{equation}\label{gradient post Gronwall1}
\| \nabla\rho \|_{L^{\infty}([0,t]; L^{1})} \leq \| \nabla\rho_0 \|_{L^{1}}+\left (C\| H \|\|\nabla\mathcal{K}\|_{L^2}\nu_{min}^{-1/2}\| \rho_0 \|_{L^{2}}F(t) \right)\| \nabla\rho \|_{L^{\infty}([0,t]; L^{1})},
\end{equation}   
with $F(t)$ as in (\ref{F}), but with $\eta$ replaced by $\eta_d$ as defined in (\ref{visccase1}).  Arguing as in the proof of Lemma \ref{sup bounds lemma}, we conclude that
\begin{equation*}
\begin{split}
&\qquad\qquad \| \nabla\rho \|_{L^{\infty}([0,T];L^{1})} \leq \| \nabla\rho_0 \|_{L^{1}}\\
&+\left (C(C_1\eta_d)^{-1/2}\| H \|\|\nabla\mathcal{K}\|_{L^2}\nu_{min}^{-1/2}\| \rho_0 \|^{1-2/d}_{L^{2}} \right)\| \nabla\rho \|_{L^{\infty}([0,T];L^{1})}.
\end{split}
\end{equation*}
We set $C_1=\|\rho_0\|_{L^1}^{-4/d}$, 
which yields the desired estimate.
\end{proof}

\textbf{Proof of Theorem \ref{smallHexistencesingle}} To prove Theorem \ref{smallHexistencesingle},  we will apply an inductive bootstrapping argument to show that the $L^2$-norm of $\rho$ is non-increasing in time.  We will then be able to use Theorems \ref{Mild Solution Proof} (with $p=2$) and \ref{small data result} to bootstrap in time and obtain a global solution.

For each integer $n\geq 0$, we will show that if $\rho$ is a smooth solution to (\ref{aggregation equation}) on $[0,nT]$ satisfying (\ref{L2decaysingle}) for all $t\in [0,nT]$, then $\rho$ can be extended to a smooth solution to (\ref{aggregation equation}) on $[0,(n+1)T]$ satisfying (\ref{L2decaysingle}) for all $t\in [0,(n+1)T]$. 

First note that the base case $n=0$ follows from Theorem \ref{Mild Solution Proof}, Theorem \ref{small data result}, our assumption on $\nu$, and Lemma \ref{energydecreasesingle}.  In particular, our assumption on $\nu$ ensures the existence of a smooth solution to (\ref{aggregation equation}) on $[0,T]$ satisfying (\ref{L2decaysingle}) on $[0,T]$.

Now assume that $\rho$ is a smooth solution to (\ref{aggregation equation}) on $[0,nT]$ for some $n\geq 1$ satisfying (\ref{L2decaysingle}) for all $t\in[0,nT]$.  We claim that for all $t\in [nT, (n+1)T]$,
\begin{equation}\label{decayinduction}
\|\rho(t)\|_{L^2} \leq \frac{\|\rho_{nT}\|_{L^2}}{\left(1+ C_1\eta t\|\rho_{nT}\|^{4/d}_{L^2}\right)^{d/4}} \leq \frac{\|\rho_{0}\|_{L^2}}{\left(1+ C_1\eta t\|\rho_{0}\|^{4/d}_{L^2}\right)^{d/4}}.
\end{equation}
To see this, note that the conditions of part 2 of Lemma \ref{sup bounds lemma} are satisfied on $[0,nT]$. 
 In particular, (\ref{sup bound 2}) holds.  By our assumption on $\nu$ in Theorem \ref{smallHexistencesingle} (see also Remark \ref{constantlessthan1}),$$
\norm{ \rho (nT)}{L^{\infty}} \leq 2\norm{ \rho_0}{L^{\infty}}.
$$
Again by our assumption on $\nu$ in Theorem \ref{smallHexistencesingle} (with $C>4$), $\nu$ satisfies the conditions of Lemma \ref{energydecreasesingle} with initial data $\rho(nT)$.  Therefore, from Lemma \ref{energydecreasesingle}, we conclude that (\ref{decayinduction}) (and thus (\ref{L2decaysingle})) holds for all $t\in [nT, (n+1)T]$.  It follows from induction, Theorem \ref{Mild Solution Proof}, and Theorem \ref{small data result} that $\rho$ can be extended to a global-in-time solution of (\ref{aggregation equation}) satisfying $\rho \in Lip([0,\infty);W^{k,1}\cap W^{k,\infty}(\R^d))$.  This completes the proof of Theorem \ref{smallHexistencesingle}.

\textbf{Proof of Theorem \ref{smallHexistence}} We complete this section with a proof of Theorem \ref{smallHexistence}.  The proof of Theorem \ref{smallHexistence} is very similar to that of Theorem \ref{smallHexistencesingle}.  The main difference is that, rather than utilizing Lemma \ref{sup bounds lemma} to control the growth of $\|\rho(t)\|_{L^{\infty}}$ with $t$, we must invoke Lemma \ref{gradient bounds lemma} to control $\| \nabla \rho (t)\|_{L^1}$.

We begin with a proof of Case 1.  We again apply an inductive bootstrapping argument to show that, for each integer $n\geq 0$, if $\rho$ is a smooth solution to (\ref{PDE}) satisfying (\ref{L2decay}) for all $t\in [0,nT]$, then $\rho$ can be extended to a smooth solution on $[0,(n+1)T]$, and (\ref{L2decay}) holds for all $t\in [0,(n+1)T]$. 

As in the case of a single species, the base case $n=0$ follows from Theorem \ref{Mild Solution Proof}, Theorem \ref{small data result}, our assumption on $\nu$, and Lemma \ref{energydecrease}.

Now assume that $\rho$ is a smooth solution to (\ref{PDE}) on $[0,nT]$ for some $n\geq 1$ satisfying (\ref{L2decay}) for all $t\in[0,nT]$.  We will show that (\ref{L2decay}) holds for all $t\in [nT, (n+1)T]$.  To see this, note that (\ref{gradient bound 33}) holds for all $t\in[0,nT]$.
The assumption (\ref{viscosityassumptionmain}) on $\nu$ in Theorem \ref{smallHexistence} ensures that 
\begin{equation}\label{controlcase1}
\norm{ \nabla\rho (nT)}{L^{1}} \leq 2\norm{ \nabla\rho_0}{L^{1}}.
\end{equation}
These estimates imply that $\nu$ satisfies the conditions of Lemma \ref{energydecrease} (Case 1) with initial data $\rho(nT)$.  Therefore, as in the proof of Theorem \ref{smallHexistencesingle}, we use Lemma \ref{energydecrease} to conclude that (\ref{L2decay}) holds for all $t\in [nT, (n+1)T]$.  It follows from induction, Theorem \ref{Mild Solution Proof}, and Theorem \ref{small data result} that $\rho$ can be extended to a global-in-time solution of (\ref{PDE}) satisfying $\rho \in Lip([0,\infty);W^{k,1}\cap W^{k,\infty}(\R^d))$.  

The proof of Case 2 of Theorem \ref{smallHexistence} is similar, but we use conservation of mass in place of Lemma \ref{gradient bounds lemma}.  This completes the proof of Theorem \ref{smallHexistence}.


\section{Proofs of Theorems \ref{comparison1} and \ref{comparison2}}\label{comparisons}

Throughout this section, we assume that the interaction kernel $\mathcal{K}$ is ideal (see Definition \ref{ideal}) and $N\geq 2$.  

When proving Theorems \ref{comparison1} and \ref{comparison2}, we make use of the fractional Laplacian, defined as follows.	 
\begin{definition}\label{fractional laplacian}
Let $0<s<1$.  The fractional Laplacian $\Lambda^{2s}$ is defined as
$$
\Lambda^{2s} u(x) = c_{d,s} \text{P.V.} \int_{\R^d} \frac{u(x)-u(y)}{\abs{x-y}^{d+2s}}\,dy = \lim_{\eps \to 0^+} c_{d,s}\int_{\R^d \setminus B_{\eps}(x)}\frac{u(x)-u(y)}{\abs{x-y}^{d+2s}}\,dy .
$$
\end{definition}
The following lemma and its proof can be found in \cite{Stinga}.
\begin{lemma}\label{full Laplace}
If $f\in C^2_B(\R^d)$, then for all $x\in\R^d$,
$$\lim_{s\rightarrow 1^-} \Lambda^{2s}f(x) = -\Delta f(x).$$ 
\end{lemma}
The proofs of Theorems \ref{comparison1} and \ref{comparison2} will also require a bound on $\|\nabla\rho\|_{L^{\infty}([0,T]; L^{\infty})}$.  We establish this bound in Lemma \ref{nablarhoinfty} below.  The proof of Lemma \ref{nablarhoinfty} is very similar to the proofs of Lemmas \ref{sup bounds lemma} and \ref{gradient bounds lemma}, so we omit some of the details.
\begin{lemma}\label{nablarhoinfty}
Let $T>0$ be arbitrary.  If $\rho\in Lip([0,T]; W^{k,1}\cap W^{k,\infty}(\R^d))$ is a solution to (\ref{PDE}) on $[0,T]$, and if (\ref{L2decay}) holds for all $t\in [0,T]$, then there exists a constant $C>0$ such that 
\begin{equation}\label{gradient bound infty}
\begin{split}
& \qquad\qquad\|  \rho \|_{L^{\infty}([0,T];W^{1,\infty})} \leq \| \rho_0 \|_{W^{1,\infty}}\\
& + \left (C\| H \|\nu_{min}^{-1/2}\eta_d^{-1/2}\|\rho_0\|^{2/d}_{L^1}\|\rho_0\|^{1-2/d}_{L^{2}} (\|\nabla\mathcal{K}\|_{L^2}+ \|\Delta\mathcal{K}\|_{L^2}) \right)\|  \rho \|_{L^{\infty}([0,T];W^{1,\infty})},
 \end{split}
\end{equation}
where $\eta_d$ is as in (\ref{visccase2}).
\end{lemma}
\begin{proof}
We use the mild formulation, H\"older's inequality, and Proposition \ref{semigroup bounds} to conclude that for each $t\in [0,T]$,
\begin{equation}\label{grad bound 33}
\begin{split}
&\| \nabla\rho(t) \|_{L^{\infty}} \leq \| \nabla\rho_0 \|_{L^{\infty}} + \\ C\nu_{min}^{-1/2}
&\int_0^t (t-\tau)^{-1/2}(\norm{\rho(\tau)}{L^{\infty}}\norm{\nabla \cdot u (\tau)}{L^{\infty}}+\norm{u(\tau)}{L^{\infty}}\norm{\nabla \rho (\tau)}{L^{\infty}})\, d\tau.
\end{split}
\end{equation}
Moreover,
\begin{equation}\label{grad bound 3333}
\begin{split}
&\| \rho(t) \|_{L^{\infty}} \leq \| \rho_0 \|_{L^{\infty}} + C\nu_{min}^{-1/2}\int_0^t (t-\tau)^{-1/2}\norm{\rho(\tau)}{L^{\infty}}\norm{ u (\tau)}{L^{\infty}} \, d\tau.
\end{split}
\end{equation}
Combining (\ref{grad bound 33}) and (\ref{grad bound 3333}) gives
\begin{equation}\label{grad final}
\begin{split}
&\qquad \| \rho(t) \|_{W^{1,\infty}} \leq \| \rho_0 \|_{W^{1,\infty}}\\
&+ C\nu_{min}^{-1/2}\int_0^t (t-\tau)^{-1/2}\norm{\rho(\tau)}{W^{1,\infty}}\left(\norm{ u (\tau)}{L^{\infty}} + \|\nabla\cdot u(\tau)\|_{L^{\infty}}\right)\, d\tau. 
\end{split}
\end{equation}
By Young's inequality, we have 
\begin{equation*}
\begin{split}
&\| u (\tau) \|_{L^{\infty}} \leq \| H \| \| \nabla\mathcal{K}\|_{L^2}\| \rho(\tau) \|_{L^{2}},\\
&\norm{\nabla \cdot u(\tau)}{L^{\infty}} \leq \| H \| \| \Delta\mathcal{K}\|_{L^2} \norm{\rho(\tau)}{L^{2}}.
\end{split}
\end{equation*}
Substituting these bounds into (\ref{grad final}) yields 
\begin{equation*}
\begin{split}
&\qquad\qquad  \|  \rho \|_{L^{\infty}([0,t];W^{1,\infty})}\leq \| \rho_0 \|_{W^{1,\infty}}\\
&+ \left (\int_0^t C\nu_{min}^{-1/2}(t-\tau)^{-1/2}\| H \|(\| \nabla\mathcal{K}\|_{L^2} +\|\Delta\mathcal{K}\|_{L^2})\| \rho(\tau) \|_{ L^{2}} \, d\tau \right) \|  \rho \|_{L^{\infty}([0,t];W^{1,\infty})}.
\end{split}
\end{equation*}
It follows from (\ref{L2decay}) that 
\begin{equation}\label{infty gradient post Gronwall}
\begin{split}
&\qquad\qquad \| \rho \|_{L^{\infty}([0,t];W^{1,\infty})} \leq \| \rho_0 \|_{W^{1,\infty}} \\
&  + \left (C \nu_{min}^{-1/2} \| H \|(\|\nabla\mathcal{K}\|_{L^2}+\|\Delta\mathcal{K}\|_{L^2})\| \rho_0 \|_{L^{2}}F(t) \right)\| \rho \|_{L^{\infty}([0,t];W^{1,\infty})},
 \end{split}
\end{equation}  
with $F(t)$ as in (\ref{F}), and $\eta_d$ as in (\ref{visccase2}).  We conclude as in the proofs of Lemmas \ref{sup bounds lemma} and \ref{gradient bounds lemma} that 
\begin{equation*}
\begin{split}
&\qquad\qquad\| \rho \|_{L^{\infty}([0,T];W^{1,\infty})} \leq \| \rho_0 \|_{W^{1,\infty}}\\
& +\left (C\nu_{min}^{-1/2}\| H \|(C_1\eta_d)^{-1/2}\| \rho_0 \|^{1-2/d}_{L^{2}} (\|\nabla\mathcal{K}\|_{L^2}+ \|\Delta\mathcal{K}\|_{L^2}) \right)\| \rho \|_{L^{\infty}([0,T];W^{1,\infty})}.
\end{split}
\end{equation*}
To complete the proof, we set $C_1 = \|\rho_0\|_{L^1}^{-4/d}$. 
\end{proof}
\noindent \textbf{Proof of Theorem \ref{comparison1}}. 
We apply a strategy similar to that used in the proof of Lemma 2.5 of \cite{LiRodrigo}.  For each $(x,t)\in \R^d\times [0,T]$, set 
\begin{equation}\label{vdef}
v(x,t) = \theta(x,t)e^{-3Mt},
\end{equation}
where $M>0$ will be chosen later. We will show that $v(x,t)\geq 0$ on $\R^d\times [0,T]$.  Assume for contradiction that there exists $(x,t)\in \R^d\times [0,T]$ such that $v(x,t) < 0$. Since $\theta$ belongs to $Lip([0,T];W^{k,1}\cap W^{k,\infty}(\R^d))$, $v$ must belong to the same space.  By the Sobolev embedding $W^{3,p}(\R^d) \hookrightarrow C^2_B(\R^d)$ for $p\in (d,\infty)$, $v$ belongs to $Lip( [0,T]; C^1_B(\R^d))$.  Assume the infimum of $v$ is $-\delta<0$. We claim that $v$ attains its infimum; that is, we claim there exists $(x^*,t^*)$ such that $v(x^*,t^*)=-\delta$.  To see this, note that $v$ belongs to $L^1(\R^d\times [0,T])$ and has bounded derivatives, implying that $v$ must vanish at infinity.  But if $v$ did not attain its minimum, then one could construct a sequence $(x_n, t_n)$ with $|x_n|\rightarrow \infty$ and $|t_n|\rightarrow \infty$ and such that $v(x_n, t_n)\rightarrow -\delta$, contradicting the fact that $v$ vanishes at infinity.  This shows that there exists $(x^*,t^*)$ such that $$v(x^*,t^*)=-\delta<0.$$ 

Now note that, given this fixed $t^*$, the minimum value of $\theta(\cdot,t^*)$ on $\R^d$ is achieved at $x^*$ by (\ref{vdef}).     

Observe that $\theta$ must satisfy  
\begin{equation}
\label{PDE 1 multispecies divergence}
\partial_t \theta + \nabla \cdot (u_1\theta) - \nu\Delta\theta = \sum_{i=1}^N \nabla \cdot (\theta_i(u_1-u_i)).
\end{equation}
We evaluate the right-hand-side of (\ref{PDE 1 multispecies divergence}) at $(x^*, t^*)$, supressing the $t^*$ dependence to simplify notation.  This gives 
\begin{equation}\label{RHSexpansion}
\begin{split}
&\sum_{i=1}^N \nabla \cdot (\theta_i(u_1-u_i))(x^*) = \sum_{i=1}^N \nabla \cdot (\theta_i(\nabla\mathcal{K}_1\ast\theta - \nabla\mathcal{K}_i\ast \gamma_i\theta))(x^*)\\
&= \sum_{i=1}^N \nabla\theta_i(x^*)\cdot (\nabla\mathcal{K}_1\ast\theta - \nabla\mathcal{K}_i\ast \gamma_i\theta)(x^*) + \sum_{i=1}^N \theta_i(x^*) (\Delta\mathcal{K}_1\ast\theta - \Delta\mathcal{K}_i\ast \gamma_i\theta)(x^*)\\
& = -\sum_{i=1}^N \nabla\theta_i(x^*)\cdot   \nabla\mathcal{K}_i\ast \gamma_i\theta(x^*) - \sum_{i=1}^N \theta_i(x^*)   \Delta\mathcal{K}_i\ast \gamma_i\theta(x^*) + \theta(x^*) \Delta\mathcal{K}_1\ast\theta(x^*),\\
\end{split}
\end{equation}
where we used that $\theta$ is minimized at $x^*$ to get the third equality. 

Again since $\theta$ achieves its minimum at $x^*$, and since $\mathcal{K}_i$ is ideal for $1\leq i \leq N$, it follows that for every $s\in(1/2,1)$ and for $C_{\mathcal{K}}$ as in Definition \ref{ideal},    
\begin{equation*}
\begin{split}
&\left|\int_{\R^d} \nabla \mathcal{K}_i(x^*-y) (\theta(y)-\theta(x^*))\, dy\right| \\
&\qquad\leq
\int_{\R^d}\abs{\nabla \mathcal{K}_i(x^*-y)}(\theta(y)-\theta(x^*))\, dy
\leq
-C_{\mathcal{K}}\Lambda^{2s}\theta(x^*).
\end{split}
\end{equation*}
It follows that for every $j$, $1\leq j\leq d$, 
\begin{equation}\label{frac Laplace bound}
C_{\mathcal{K}}\Lambda^{2s}\theta(x^*) \leq -\int_{\R^d} \partial_{x_j} \mathcal{K}_i(x^*-y) (\theta(y)-\theta(x^*))\, dy \leq
-C_{\mathcal{K}}\Lambda^{2s}\theta(x^*).
\end{equation}
By an identical argument, we also have
\begin{equation}\label{frac Laplace bound2}
\begin{split}
C_{\mathcal{K}}\Lambda^{2s}\theta(x^*) \leq -\int_{\R^d} \Delta \mathcal{K}_i(x^*-y) (\theta(y)-\theta(x^*))\, dy \leq
-C_{\mathcal{K}}\Lambda^{2s}\theta(x^*)
\end{split}
\end{equation}
for all $i$ between $1$ and $N$.  By Lemma \ref{full Laplace}, passing to the limit as $s\rightarrow 1^-$ in (\ref{frac Laplace bound}) and (\ref{frac Laplace bound2}) yields the inequalities
\begin{equation}\label{Laplace bound}
-C_{\mathcal{K}}\Delta\theta(x^*) \leq -\int_{\R^d} \partial_{x_j} \mathcal{K}_i(x^*-y) (\theta(y)-\theta(x^*))\, dy \leq
C_{\mathcal{K}}\Delta\theta(x^*),
\end{equation}
and
\begin{equation}\label{Laplace bound2}
\begin{split}
-C_{\mathcal{K}}\Delta\theta(x^*) \leq -\int_{\R^d} \Delta \mathcal{K}_i(x^*-y) (\theta(y)-\theta(x^*))\, dy \leq
C_{\mathcal{K}}\Delta\theta(x^*)
\end{split}
\end{equation}
for all $i$ between $1$ and $N$.

Set 
\begin{equation*}
\begin{split}
&G(x,t)
=
-\sum_{i=1}^N\left( \gamma_i\nabla\theta_i(x,t)\cdot \int_{\R^d}\nabla \mathcal{K}_i(x-y)\, dy + \gamma_i\theta_i(x,t) \int_{\R^d} \Delta \mathcal{K}_i(x-y)\, dy\right)\\
 &\qquad \qquad + \Delta \mathcal{K}_1\ast\theta(x),
\end{split}
\end{equation*}
and set $$ \psi(x,t) = \sum_{i=1}^N C_{\mathcal{K}}\gamma_i\theta_i(x,t).$$

Reintroducing $t^*$, it follows from (\ref{PDE 1 multispecies divergence}), (\ref{RHSexpansion}), (\ref{Laplace bound}) and (\ref{Laplace bound2}) that
\begin{equation}\label{keyinequality}
\begin{split}
&(\partial_t \theta + \nabla \cdot (u_1\theta))(x^*,t^*)
\geq
(\nu-\norm{ \psi(t^*)}{W^{1,\infty}})\Delta\theta(x^*,t^*)
+ (\theta G)(x^*,t^*).
\end{split}
\end{equation}

Since $\mathcal{K}$ is ideal, for all $(x,t)\in \R^d\times [0,T]$,  
\begin{equation*}
\begin{split}
&|G(x,t)| \leq C_{\mathcal{K}}\max_i\|(1+\gamma_i)\theta_i\|_{L^{\infty}([0,T];W^{1,\infty})}(\|\nabla\mathcal{K}\|_{L^1} + \|\Delta\mathcal{K}\|_{L^1}),\\ 
& |\nabla \cdot u_1(x,t)| \leq \|\theta\|_{L^{\infty}([0,T];L^{\infty})}\|\Delta\mathcal{K}\|_{L^1}.
\end{split}
\end{equation*}
For simplicity, redefine $C_{\mathcal{K}}=\max\{1,C_{\mathcal{K}}\}$, and set $$M=C_{\mathcal{K}}\|H\|\max_i|1+\gamma_i|\|\rho\|_{L^{\infty}([0,T];W^{1,\infty})}(\|\nabla\mathcal{K}\|_{L^1} + \|\Delta\mathcal{K}\|_{L^1}).$$ By Lemma \ref{nablarhoinfty} and our assumption on $\nu$ (see also Remark \ref{constantlessthan1}),
\begin{equation*}
\begin{split}
&M \leq C_{\mathcal{K}}\|H\|\max_i \abs{1+\gamma_i}\norm{\rho_0}{W^{1,\infty}}(\|\nabla\mathcal{K}\|_{L^1} + \|\Delta\mathcal{K}\|_{L^1})<\infty.
\end{split}
\end{equation*}
Again by our assumption on $\nu$, 
\begin{equation*} 
\norm{ \psi}{L^{\infty}([0,T];W^{1,\infty})} \leq \sum_{i=1}^N C_{\mathcal{K}}  |\gamma_i h_{1i}|\| \rho_i\|_{L^{\infty}([0,T];W^{1,\infty})} \leq C_{\mathcal{K}} \|H\|  \|\rho_0 \|_{W^{1,\infty}},
\end{equation*}
where $C_{\mathcal{K}}$ now depends on $N$.  Then (\ref{keyinequality}) and our assumption on $\nu$ imply that
\begin{align*}
&(\partial_tv)(x^*,t^*) = -3Mv(x^*,t^*)+\partial_t\theta(x^*,t^*)e^{-3Mt^*} \\
&\qquad \geq
-3Mv(x^*,t^*) -\nabla\cdot(u_1\theta)(x^*,t^*)e^{-3Mt^*}\\
&\qquad +(\nu-\norm{\psi}{L^{\infty}([0,T];W^{1,\infty})})\Delta v(x^*,t^*)- (v G)(x^*,t^*)\\
&=
-(3M+G(x^*,t^*))v(x^*,t^*)-(v\nabla\cdot u_1)(x^*,t^*) + (\nu-\norm{\psi}{L^{\infty}([0,T];W^{1,\infty})})\Delta v(x^*,t^*)\\
&\qquad \geq
-v(x^*,t^*)(3M+G(x^*,t^*)+(\nabla\cdot u_1)(x^*,t^*))
\geq M\delta.
\end{align*}
This contradicts the fact that $v$ is minimized at $(x^*,t^*)$.  It follows that $\theta(x,t) \geq 0$ for all $(x,t) \in \R^d\times[0,T]$. This completes the proof of Theorem \ref{comparison1}. \\

\noindent\textbf{Proof of Theorem \ref{comparison2}} We now turn our attention to Theorem \ref{comparison2}.  The proof is similar to that of Theorem \ref{comparison1}.  Assume $\rho_1$ and $\rho_2$ represent the density of two species satisfying the assumptions of Theorem \ref{comparison2}.  Set $\theta = \rho_1-c_0\rho_2$ and $\overline{u}=u_1-u_2$.  Then $\theta$ and $\overline{u}$ satisfy
\begin{equation}
\label{Two Species Model}
\begin{cases}
\partial_{t}\theta - \nu\Delta \theta + \nabla \cdot (u_2\theta) = -\nabla \cdot (\overline{u}\rho_1)\\
\overline{u} =  ( h_{11}\nabla\mathcal{K}_{11} -h_{21}\nabla \mathcal{K}_{21})\ast \rho_1 + (h_{12}\nabla\mathcal{K}_{12}-h_{22}\nabla\mathcal{K}_{22})\ast \rho_2\\
\theta \vert_{t=0} = \theta_0(x).
\end{cases}
\end{equation}
For each $(x,t)\in \R^d\times [0,T]$, set $$v(x,t) = \theta(x,t)e^{-Mt},$$ where $M>0$ will be chosen later.  Assume for contradiction that the infimum of $v$ on $\R^d\times [0,T]$ is negative.  Denote this infimum by $-\delta$.  As in the proof of Theorem \ref{comparison1}, $v$ must attain its infimum, so there exists $(x^*,t^*) \in \R^d\times [0,T]$ such that $v(x^*,t^*)=-\delta<0$.  Given this fixed $t^*$, the minimum value of $\theta(\cdot,t^*)$ on $\R^d$ must be achieved at $x^*$.  

Set
$$ \mathcal{K}_H =  h_{11}\mathcal{K}_{11} -h_{21} \mathcal{K}_{21}$$ and note that
\begin{equation}\label{ineq1}
\begin{split}
&\overline{u}(x^*) = \nabla \mathcal{K}_H \ast \theta(x^*)\\
&=\int_{\R^d} \nabla \mathcal{K}_H(x^* -y)(\theta(y)-\theta(x^*)) \, dy + \theta(x^*)\int_{\R^d}\nabla \mathcal{K}_H(x^*-y) \, dy.
\end{split}
\end{equation}
Since $\mathcal{K}_H$ is ideal, and since $\theta$ is minimized at $x^*$, as in the proof of Theorem \ref{comparison1}, we have that 
\begin{equation*}\label{Laplace bound 2}
\begin{split}
\left|\int_{\R^d} \nabla \mathcal{K}_H(x^*-y) (\theta(y)-\theta(x^*))\, dy\right| \leq C_{\mathcal{K}}\|H\|\Delta\theta(x^*).
\end{split}
\end{equation*}
Also note that 
\begin{align*}
&\nabla \cdot \overline{u}(x^*) = \int_{\R^d} \Delta \mathcal{K}_H(x^*-y)(\theta(y)-\theta(x^*)) \,dy + \theta(x^*)\int_{\R^d} \Delta \mathcal{K}_H(x^*-y) \, dy,
\end{align*}
and again since $\mathcal{K}_H$ is ideal,
$$  \left | \int_{\R^d} \Delta \mathcal{K}_H(x^*-y)(\theta(y)-\theta(x^*)) \,dy \right | \leq C_{\mathcal{K}}\|H\|\Delta\theta(x^*).$$
Substituting (\ref{ineq1}) into $(\ref{Two Species Model})_1$ and applying the above inequalities gives 
\begin{equation}\label{ineq2}
\begin{split}
&\partial_t\theta(x^*) + \nabla \cdot (u_2\theta)(x^*) \geq (\nu-C_{\mathcal{K}}\|H\|\norm{ \rho_1}{W^{1,\infty}})\Delta\theta(x^*) \\ 
&\qquad- (\theta \nabla \rho_1)(x^*)\cdot\int_{\R^d}\nabla \mathcal{K}_H(x^*-y)\, dy - (\theta \rho_1)(x^*)\int_{\R^d}\Delta \mathcal{K}_H(x^*-y)\, dy\\
&\geq (\nu-C_{\mathcal{K}}\|H\|\norm{ \rho}{W^{1,\infty}})\Delta\theta (x^*)\\ &\qquad- \theta(x^*)\left(\nabla \rho_1(x^*)\cdot\int_{\R^d} \nabla \mathcal{K}_H(x^* - y)\, dy + \rho_1(x^*) \int_{\R^d} \Delta \mathcal{K}_H(x^*-y)\, dy\right).
\end{split}
\end{equation}
Similar to the proof of Theorem \ref{comparison1}, set 
$$
G(x,t) = \rho_1(x,t)\int_{\R^d} \Delta \mathcal{K}_H(x - y)\, dy + \nabla \rho_1(x,t)\cdot \int_{\R^d} \nabla \mathcal{K}_H(x-y)\, dy.
$$
Again since $\mathcal{K}_H$ is ideal, for all $(x,t)\in \R^d\times[0,T]$,
\begin{equation*}
\begin{split}
&|G(x,t)| \leq C_{\mathcal{K}}\|H\|(\|\nabla\mathcal{K}\|_{L^1} + \| \Delta\mathcal{K}\|_{L^1})\|\rho\|_{L^{\infty}([0,T];W^{1,\infty})}, \\
&|\nabla \cdot u_2(x,t)| \leq \|H\|\|\Delta\mathcal{K}\|_{L^1}\|\rho\|_{L^{\infty}([0,T];L^{\infty})}.
\end{split}
\end{equation*}
Now redefine $C_\mathcal{K}=\max\{1,C_{\mathcal{K}}\}$ and set $$ M = C_{\mathcal{K}}\|H\|(\|\nabla\mathcal{K}\|_{L^1} + \| \Delta\mathcal{K}\|_{L^1})\|\rho\|_{L^{\infty}([0,T];W^{1,\infty})}.$$  By Lemma \ref{nablarhoinfty} and our assumption on $\nu$,
\begin{equation*}
\begin{split}
&M \leq C_{\mathcal{K}}\|H\|(\|\nabla\mathcal{K}\|_{L^1} + \| \Delta\mathcal{K}\|_{L^1})\|\rho_0\|_{W^{1,\infty}}<\infty.
\end{split}
\end{equation*}
It follows again by our assumption on $\nu$ that
\begin{align*}
&(\partial_tv)(x^*,t^*) = -3Mv(x^*,t^*)+\partial_t\theta(x^*,t^*)e^{-3Mt^*} \\
&\geq
-3Mv(x^*,t^*) -\nabla\cdot(u_2\theta)(x^*,t^*)e^{-3Mt^*}\\
&\qquad + (\nu-C_{\mathcal{K}}\|H\|\norm{\rho_0}{W^{1,\infty}})\Delta v(x^*,t^*)- (v G)(x^*,t^*)\\
&=
-(3M+G(x^*,t^*))v(x^*,t^*)-(v\nabla\cdot u_2)(x^*,t^*) + (\nu-C_{\mathcal{K}}\|H\|\norm{\rho_0}{W^{1,\infty}})\Delta v(x^*,t^*)\\
&=
-v(x^*,t^*)(3M+G(x^*,t^*)+\nabla\cdot u_2(x^*,t^*))
\geq M\delta.
\end{align*}
This contradicts the fact that $v$ is minimized at $(x^*,t^*)$.  It follows that $\theta(x,t) \geq 0$ for all $(x,t) \in \R^d\times[0,T]$.  This completes the proof of Theorem \ref{comparison2}.   

\section{Concluding Remarks}\label{Concluding Remarks}

Above, we established global existence of smooth solutions to (\ref{PDE}) for a class of singular kernels with a smallness assumption on the initial data, and with no assumptions placed on the interaction matrix $H$.  In addition, we established a type of balance on the species interaction, via assumptions on $H$, which implies that pointwise estimates comparing species densities are preserved by the evolution.  A remaining open problem is to establish a reasonable set of assumptions on the interaction kernel which, when combined with these assumptions on $H$, yields global existence of solutions.  Such a result would perhaps complement those in \cite{JPZ, CSS}, in which the so-called balance condition is placed on the interaction kernels.  It is worth noting, however, that the assumptions on $H$ in Theorems \ref{comparison1} and \ref{comparison2} do not satisfy the balance condition, so they would give a sort of different balance condition from that in the literature, albeit with a smallness assumption on the initial data.

Note further that if, in addition to the assumptions of Theorem \ref{comparison1}, we also assume that $\Delta\mathcal{K}\geq 0$ a.e., then the assumptions on $H$ imply that the divergence of the initial velocity is non-negative.  Our preliminary calculations suggest that non-negative divergence of velocity will be preserved over time, yielding global existence of solutions.  However, kernels satisfying the assumptions of Definition \ref{ideal} (even slightly weakened, with $\nabla\mathcal{K}$ and $\Delta\mathcal{K}$ in $L^1(\R^d)$) and the assumption that $\Delta\mathcal{K}\geq 0$ are extremely difficult to construct.    If one can adequately relax these assumptions, one might obtain global existence of solutions for a sizable set of interaction kernels under the assumptions on $H$ in Theorem \ref{comparison1}. 

\appendix
\section{A Priori Estimates}\label{A Priori Estimates}
Throughout Appendix \ref{A Priori Estimates}, unless otherwise stated, we assume that 
\begin{itemize}
\item{$N\geq 2$}
\item{$\mathcal{K}$ is admissible}
\item{$\rho$ is a solution of \eqref{PDE} in $Lip([0,T]; W^{k,1} \cap W^{k,\infty}(\R^d))$ for $k\geq 3$.} 
\end{itemize}
Under these assumptions, we establish conservation of mass and show that the solution remains non-negative on $[0,T]$ given non-negative initial data.  The results in this section are well-known for solutions to (\ref{aggregation equation}); i.e., the case $N=1$, and we apply them as needed throughout the paper (with $\|H\|=1$) to establish the single-species results. 
 The extensions of these results to the multi-species case are straightforward.  We include them here for completeness.

\subsection{Positivity of Solution} 

The goal of this subsection is to establish positivity of solutions to (\ref{PDE}).  The following lemma and its proof can be found in \cite{LiRodrigo}. In \cite{LiRodrigo}, the authors require $\phi \in C^{1}((0,T];C^2(\R^d))$, but this can be weakened to $\phi\in W^{1,\infty}((0,T];C^2(\R^d))$ as we only need to bound $\sup_{\Omega_T}\abs{\partial_t\phi}$.

Below, to simplify notation, we say $f \in C_{t,x}^{r,s}$ if $f(\cdot,t) \in C^s(\R^d)$ and $f(x,\cdot) \in C^r(X)$, where $X$ denotes the time interval under consideration.  
\begin{lemma}
\label{positivity lemma}
Let $T>0$.  Set $\Omega_T =\R^d\times (0,T]$. Let $\phi$ be a given function in $W^{1,\infty}((0,T];C^2(\R^d)) \cap C_{t,x}^0(\overline{\Omega_T}) \cap L_{t,x}^p(\Omega_T)$ for some $p \in [1,\infty)$. Assume $\rho_0: \R^d \to \R$, $w: \Omega_T \to \R^d$ are also given functions with $\rho_0\in C(\R^d)$ and $w\in C_{t,x}^{0,1}(\Omega_T)$. Further assume that
\begin{itemize}
\item 
on $\Omega_T$, $\phi$ satisfies the point-wise estimate
\begin{equation*}
\begin{cases}
    \partial_t \phi - \nu\Delta\phi \geq - \nabla \cdot(w\phi) \\
    \phi\vert_{t=0} = \rho_0(x)
    \end{cases}
\end{equation*}

\item 
there exists $M_1 \geq 0$ such that
$$
\sup_{\Omega_T} \{ \abs{\partial_t \phi} + \abs{\nabla\phi} + \abs{\nabla^2 \phi} \} + \sup_{\overline{\Omega_T}} \abs{\phi} \leq M_1 < \infty
$$

\item 
$\rho_0(x) \geq 0$ and there exists $M_2 \geq 0$ such that 
$$
\sup_{\overline{\Omega_T}} \abs{\nabla \cdot w} \leq M_2 < \infty.
$$
\end{itemize}
Under the above assumptions, $\phi(x,t)\geq 0 $ on $\overline{\Omega_T}$.
\end{lemma}

Before applying the Positivity Lemma to \eqref{PDE}, we must establish an upper bound on the divergence of the velocity $u$.

\begin{lemma}
\label{div of velocity is bounded}
Let $T>0$.  If $\rho$ is a solution to \eqref{PDE} with corresponding velocity $u$, then 
$$
\|\nabla \cdot u(t) \|_{L^{\infty}(\R^d\times[0,T])}\leq C\|H\|\|\nabla\mathcal{K}\|_{L^2}\norm{\nabla \rho}{L^{\infty}([0,T];L^1 \cap L^{\infty})}.
$$
\end{lemma}
\begin{proof}
Note that for species $j$ with $u_j =$ ($u_j^1, u_j^2, ... u_j^d$), for each $t\in [0,T]$, 
\begin{equation*}
\begin{split}
& \|\nabla \cdot u_j(t) \|_{L^{\infty}(\R^d\times[0,T])} = \norm{\sum_{i=1}^d \partial_{x_i} u^i_j(t)}{L^{\infty}} \leq \sum_{i=1}^d \norm{\partial_{x_i} u^i_j(t)}{L^{\infty}}\\
&\qquad \leq C\norm{\nabla u_j(t)}{L^{\infty}} \leq C\|H\|\|\nabla\mathcal{K}\|_{L^2}\norm{\nabla \rho(t)}{L^1 \cap L^{\infty}}
\end{split}
\end{equation*}
by Lemma \ref{ConstLawEstimate} with $f=\nabla\rho$.  Taking the supremum over all $t \in [0,T]$ and over all $j$, $1\leq j \leq N$, gives 
$$
\|\nabla \cdot u(t) \|_{L^{\infty}(\R^d\times[0,T])} \leq C\|H\|\|\nabla\mathcal{K}\|_{L^2}\norm{\nabla \rho}{L^{\infty}([0,T]; L^1 \cap L^{\infty})},
$$
which is what we desired to show.
\end{proof}

\begin{corollary}
If $\rho$ is a solution to \eqref{PDE} satisfying $\rho_0(x) \geq 0$ for all $x \in \R^d$, then $\rho(x,t) \geq 0$ on $\R^d\times[0,T]$.
\end{corollary}
\begin{proof} This follows from Lemmas \ref{positivity lemma} and \ref{div of velocity is bounded}.
\end{proof}

\subsection{Conservation of Mass}
We establish conservation of mass for \eqref{PDE} in the following theorem.
\begin{theorem} Let $T>0$.  If $\rho$ is a solution to \eqref{PDE} with $\rho_0(x)\geq 0$, then
$$
\norm{\rho(t)}{L^1} = \norm{\rho_0}{L^1} \text{ for all } t\in [0,T].
$$
\end{theorem}
\begin{proof} We show that 
\begin{equation}\label{conservation}
\frac{d}{dt}\norm{\rho(t)}{L^1} = \frac{d}{dt} \int_{\R^d} \rho(x,t)\,dx=0 \text{ for all } t\in [0,T].
\end{equation}
Note that for all $t\in[0,T]$, $u(t)$ belongs to $W^{3,\infty}(\R^d)$ by Lemma \ref{ConstLawEstimate}.  Thus $(u\rho)(t)$ and $\nabla\rho(t)$ belong to $W^{1,1}(\R^d)$ for every $t\in [0,T]$, from which it follows that 
\begin{equation*}
\int_{\R^d} \nabla\cdot (u\rho) \, dx= 0 \text { and } \int_{\R^d} \Delta \rho \, dx= 0.
\end{equation*}
The equality (\ref{conservation}) then follows from integrating \eqref{PDE} over $\R^d$.  
\end{proof}


\section{The Mild Solution}\label{The Mild Solution}
In this section of the Appendix, we again assume $\mathcal{K}$ is admissible, and we establish short-time existence and uniqueness of a mild solution to \eqref{PDE}, as in Definition \ref{Mild Solution Def}.  Our strategy is to apply a fixed point argument similar to that in \cite{Wu}.  To simplify notation, as in \cite{Wu}, we set
\begin{equation}\label{nonlinear term}
B(u,\rho)(t) = \int_0^t \nabla e^{\nu\Delta(t-\tau)}(u\rho)(\tau) \, d\tau,
\end{equation}
we fix $p\in [2,\infty]$, and we define
$$ E = L^{\infty}([0,T];L^1 \cap L^{p}(\R^d)).$$

We first establish some useful estimates.
\begin{proposition}
\label{non-linearity bound}
Let $T>0$.  If $u\in L^{\infty}([0,T];L^{\infty}(\R^d))$ and $\rho\in E$, then there exists a constant $C$ depending only on $d$ such that
$$
\norm{B(u,\rho)}{E} \leq C\nu_{min}^{-1/2}T^{\frac{1}{2}}\norm{u}{L^{\infty}([0,T];L^{\infty})}\norm{\rho}{E},
$$
\end{proposition}
\begin{proof}
An application of Young's convolution inequality, Hölder's inequality, and Proposition \ref{semigroup bounds} gives
\begin{align*}
\norm{B(u,\rho)(t)}{L^1 \cap L^{p}}
&\leq
C\nu_{min}^{-1/2}\int_0^t \abs{t-\tau}^{-\frac{1}{2}}\norm{u(\tau)}{L^{\infty}}\norm{\rho(\tau)}{L^1 \cap L^{p}} \, d\tau\\
&\leq C\nu_{min}^{-1/2}t^{\frac{1}{2}}\norm{u}{L^{\infty}([0,T];L^{\infty})}\norm{\rho}{E},
\end{align*}
which concludes the proof.
\end{proof}
We also have the following useful lemma.
\begin{lemma}\label{ConstLawEstimate}
If $f\in L^1\cap L^p(\R^d)$ for $p\geq 2$, then,
$$
\|\nabla((H\circ\mathcal{K})\ast f)\|_{L^{\infty}} \leq \|H\| \|\nabla \mathcal{K} \|_{L^2}\|f\|_{L^1\cap L^{p}}. $$
\end{lemma}
\begin{proof}
By Young's inequality
\begin{equation*}
\norm{\nabla((H\circ\mathcal{K})\ast f)}{L^{\infty}}
\leq
\|H\|\|\nabla \mathcal{K} \|_{L^2} \|f\|_{L^2} \leq \|H\|\|\nabla \mathcal{K} \|_{L^2} \|f\|_{L^1\cap L^p}.
\end{equation*}
\end{proof}
We now apply Proposition \ref{non-linearity bound} and Lemma \ref{ConstLawEstimate} to establish short-time existence and uniqueness of a mild solution.
\begin{theorem}\label{Mild Solution Proof}
Fix $p\in [2,\infty]$.  Given $\nu$, $H$, $\mathcal{K}$, and $\rho_0\in L^1\cap L^{\infty}(\R^d)$, there exists a constant $C\geq 1$, depending only on $d$, such that whenever $T>0$ satisfies 
\begin{equation}\label{smallness}
\nu_{min}^{-1/2}T^{\frac{1}{2}}\|H\|\|\nabla\mathcal{K}\|_{L^2}\norm{\rho_0}{L^1 \cap L^{p}} \leq \frac{1}{4C},
\end{equation}
one can find a unique $\rho \in E$ satisfying \eqref{mild solution} on $[0,T]$ with $\rho(0) = \rho_0$. The solution $\rho$ satisfies
$$
\norm{\rho}{E} \leq 2\norm{\rho_0}{L^1 \cap L^{p}}.
$$
Moreover, one can find a single $T$ which satisfies (\ref{smallness}) for all $p\in [2,\infty]$.
\end{theorem}
\begin{remark}\label{singleT}
We address the final statement in Theorem \ref{Mild Solution Proof} above.  To see why this statement holds, note that $$\|\rho_0\|_{L^1\cap L^p} \leq 2 \|\rho_0\|_{L^1\cap L^{\infty}}$$ for every $p\in[2,\infty]$.  Thus, if we choose $T$ to satisfy 
\begin{equation}\label{smallnessremark}
\nu_{min}^{-1/2}T^{\frac{1}{2}}\|H\|\|\nabla\mathcal{K}\|_{L^2}\norm{\rho_0}{L^1 \cap L^{\infty}} \leq \frac{1}{8C},
\end{equation}
then $T$ also satisfies (\ref{smallness}) for all $p\in[2,\infty]$.
\end{remark}
\begin{proof} (of Theorem \ref{Mild Solution Proof}) Our strategy is to apply a Banach fixed point argument.  Define the map
$$\rho\rightarrow A\rho :=  e^{\nu\Delta t}\rho_0 - \int_0^t \nabla e^{\nu\Delta(t-\tau)}(u\rho)(\tau) \, d\tau,$$
and let $R = 2\norm{\rho_0}{L^1 \cap L^{p}}$.  

We will show that there exists $C\geq 1$ such that, given (\ref{smallness}), $A$ is a strict contraction from $B_R$ into itself.  We first show that $A$ is a strict contraction.  To this end, assume $\rho$ and $\overline{\rho}$ belong to $B_R \subset E$, with $u$ and $\bar{u}$ their corresponding velocities.  By Proposition \ref{non-linearity bound} and Lemma \ref{ConstLawEstimate},
\begin{align*}
\norm{A\rho - A\overline{\rho}}{E}
&\leq
\norm{B(u-\overline{u},\rho)}{E}+\norm{B(\overline{u},\rho-\overline{\rho})}{E}\\
&\leq C\nu_{min}^{-1/2} T^{\frac{1}{2}}\left(\norm{u-\overline{u}}{L^{\infty}([0,T];L^{\infty})}\norm{\rho}{E} + \norm{\overline{u}}{L^{\infty}([0,T];L^{\infty})}\norm{\rho-\overline{\rho}}{E}\right)\\
&\leq C\|H\|\|\nabla\mathcal{K}\|_{L^2}\nu_{min}^{-1/2} T^{\frac{1}{2}}\left(\norm{\rho-\overline{\rho}}{E}\norm{\rho}{E}+ \norm{\overline{\rho}}{E}\norm{\rho-\overline{\rho}}{E}\right)\\
&\leq C\|H\|\|\nabla\mathcal{K}\|_{L^2}\nu_{min}^{-1/2}  2\norm{\rho_0}{L^1 \cap L^{p}}T^{\frac{1}{2}}\norm{\rho-\overline{\rho}}{E}.
\end{align*}
Thus, whenever $\nu_{min}^{-1/2}T^{\frac{1}{2}}\|H\|\|\nabla\mathcal{K}\|_{L^2}\norm{\rho_0}{L^1 \cap L^{p}} \leq \frac{1}{4C}$,
\begin{equation}\label{contraction}
\norm{A\rho-A\overline{\rho}}{E} \leq \frac{1}{2}\norm{\rho-\overline{\rho}}{E}.
\end{equation}
We conclude that $A$ is a strict contraction.  

We now prove that $A$ maps $B_R$ into $B_R$.  We must show that, given $\rho\in B_R$, $\norm{A\rho}{E} \leq R$.  Assume $\rho$ belongs to $B_R$, and let $0 \in E$ denote the zero element.  By Proposition \ref{fractional heat kernel bound}, (\ref{contraction}), and the definition of $R$,
\begin{equation*}
\begin{split}
&\norm{A\rho}{E} 
\leq \norm{A\rho - A0}{E} +\norm{A0}{E} \\
&\qquad \leq \frac{1}{2}\norm{\rho}{E} + \norm{\rho_0}{L^1 \cap L^{p}} \leq \frac{1}{2}\norm{\rho}{E} + \frac{R}{2} \leq R.
\end{split}
\end{equation*}
Thus, $A$ maps $B_R$ into $B_R$.  

By the Banach Fixed Point Theorem, we conclude that there exists $\rho\in B_R$ satisfying $A\rho=\rho$.
\end{proof}
\section{The Classical Solution}\label{The Classical Solution}
In this section, we again assume $\mathcal{K}$ is admissible, and we show that if $\rho$ is a mild solution of \eqref{PDE} with $\rho_0 \in W^{k,1}\cap W^{k,\infty}(\R^d)$ for some $k\geq 3$, then it is a classical solution to (\ref{PDE}).
\begin{theorem}
\label{small data result}
Let $k$ be a non-negative integer.  Given $\nu$, $H$, $\mathcal{K}$, and $\rho_0\in W^{k,1} \cap W^{k,\infty}(\R^d)$, there exists a constant $C\geq 1$, depending only on $d$, such that whenever $T>0$ satisfies 
\begin{equation}\label{smallness2}
\nu_{min}^{-1/2}T^{\frac{1}{2}}\|H\|\|\nabla\mathcal{K}\|_{L^2}\norm{\rho_0}{L^1 \cap L^{\infty}} \leq \frac{1}{8C},
\end{equation}
then one can find a unique mild solution $\rho$ of (\ref{PDE}) in $L^{\infty}([0,T];W^{k,1} \cap W^{k,\infty}(\R^d))$.

Moreover, if $k\geq 3$, then $\rho$ is a classical solution of (\ref{PDE}) in $Lip([0,T];W^{k,1} \cap W^{k,\infty}(\R^d))$.
\end{theorem}
Note that by Remark \ref{singleT}, if $T$ satisfies (\ref{smallness2}), then it also satisfies (\ref{smallness}) for all $p\in[2,\infty]$.
\begin{proof}
The proof is very similar to that in, for example, Theorem 3.5 of \cite{Wu}, so we omit many of the details.  The idea is to apply a differential operator $D$ of order one to the mild formulation (\ref{mild solution}) to obtain an equality of the form $D\rho = G(D\rho)$, and to show that $G$ is a contraction mapping from $E$ to itself.  Here, $E$ is defined via the norm
$$\|\rho\|_{E}= \|\rho\|_{L^{\infty}([0,T]; L^1\cap L^{\infty}(\R^d))}+\|D\rho\|_{L^{\infty}([0,T]; L^1\cap L^{\infty}(\R^d))}.$$  The argument that $G$ is a contraction is very similar to that in the proof of Theorem \ref{Mild Solution Proof}.  We  then conclude from the Banach Fixed Point Theorem that $G$ has a fixed point.  By uniqueness of the solution in Theorem \ref{Mild Solution Proof}, the fixed point of $G$ must agree with the fixed point established in Theorem \ref{Mild Solution Proof}.

A similar argument can be applied for derivatives of order greater than or equal to two, yielding $$D^{\alpha} \rho \in L^{\infty}([0,T]; L^1 \cap L^{\infty}(\R^d)) \text{ for }\abs{\alpha} \leq k.$$     

\Obsolete{We proceed by induction on $k$. To establish the base case $k=0$, note that $\rho \in L^{\infty} ([0,T];L^1\cap L^{\infty} (\R^d))$ by Theorem \ref{Mild Solution Proof}.  Now assume that $k\geq 1$ and
$$
\norm{D^{\alpha}\rho}{L^{\infty}([0,T];L^1\cap L^{\infty})} <\infty \text{ for all } \alpha \text{ such that }|\alpha| \leq k-1.
$$
We want to show that 
$$
\norm{D^{\alpha}\rho}{L^{\infty}([0,T];L^1\cap L^{\infty})} <\infty \text{ for all } \alpha  \text{ such that }|\alpha| \leq k.
$$
Applying $D^{\alpha}$ for $|\alpha| = k$ to (\ref{mild solution}) gives 
\begin{align*}
&\max_{|\alpha| = k}\norm{D^{\alpha}\rho(t)}{L^1 \cap L^{\infty}} 
\leq
\norm{\rho_0}{W^{k,1} \cap W^{k,\infty}}\\
&\ \qquad+ C\nu_{min}^{-1/2}\int_0^t \abs{t-\tau}^{-\frac{1}{2}}\max_{|\alpha| = k}\sum_{|\beta|\leq k}\norm{ D^{\alpha-\beta}u D^{\beta}\rho}{L^1 \cap L^{\infty}} \, d\tau\ \leq 
\norm{\rho_0}{W^{k,1} \cap W^{k,\infty}}\\
&+ C(1+\|\rho\|_{W^{k-1,1}\cap W^{k-1,\infty}})^2\nu_{min}^{-1/2}\int_0^t \abs{t-\tau}^{-\frac{1}{2}}\max_{|\alpha| = k}(1+\norm{D^{\alpha}\rho}{L^1 \cap L^{\infty}}) \, d\tau, 
\end{align*}
where we applied Leibniz rule to get the first inequality, and we applied Lemma \ref{induction lemma} and the induction hypothesis to get the second inequality.  An application of Osgood's lemma gives $$D^{\alpha} \rho \in L^{\infty}([0,T]; L^1 \cap L^{\infty}(\R^d)) \text{ for }\abs{\alpha} \leq k.$$ } 

Note also that by Lemma \ref{ConstLawEstimate}, $$ u \in L^{\infty}([0,T]; W^{k,\infty}(\R^d)).$$ 

To show that for $k\geq 3$, $\rho$ is a classical solution to (\ref{PDE}), note that the Sobolev embedding $W^{3,p}(\R^d)\hookrightarrow C^2_B(\R^d)$ for $p\in (d,\infty)$ and Lemma \ref{ConstLawEstimate} imply that both $u$ and $\rho$ belong to $L^{\infty}([0,T]; C^2_B(\R^d))$.  Taking the time derivative of (\ref{mild solution}) and using Leibniz rule gives
\begin{equation*}
\begin{split}
&\partial_t \rho(x,t) = \partial_t (e^{\nu\Delta t}\rho_0)(x) - \partial_t \int_0^t \nabla e^{\nu\Delta(t-\tau)}(u\rho)(\tau)\, d\tau \\
&\qquad = \nu\Delta(e^{\nu\Delta t}\rho_0)(x) - \nabla e^{\nu\Delta 0}(u\rho)(x,t)  - \int_0^t \partial_t [\nabla e^{\nu\Delta (t-\tau)}(u\rho)(\tau)] \, d\tau\\
&\qquad = \nu\Delta(e^{\nu\Delta t}\rho_0)(x) - \nabla\cdot (u\rho)(x,t) - \int_0^t \nu\Delta\nabla e^{\nu\Delta(t-\tau)}(u\rho)(\tau) \, d\tau\\
&\qquad =  - \nabla\cdot (u\rho)(x,t) + \nu\Delta\left(e^{\nu\Delta t}\rho_0 
 - \int_0^t \nabla e^{\nu\Delta(t-\tau)}(u\rho)(\tau) \, d\tau \right)\\
 &\qquad = - \nabla\cdot (u\rho)(x,t) + \nu\Delta\rho(x,t),
\end{split}
\end{equation*}
where we used that $\rho$ satisfies (\ref{mild solution}) to obtain the last equality.  Using the regularity of $u$ and $\rho$, we conclude that $\partial_t\rho$ exists and is bounded.  Thus, $\rho \in Lip([0,T]; C^2_B(\R^d))$ and satisfies (\ref{PDE}). 
\end{proof}

\section*{Acknowledgements}
\noindent The authors thank Patrick De Leenheer for useful discussions.  EC gratefully acknowledges support by the Simons Foundation through Grant No. 429578.   


\end{document}